\documentclass{imsart}
\usepackage{amsmath}
\usepackage{amsthm}
\usepackage{amsfonts}
\usepackage{bbm}
\usepackage{color}
\theoremstyle{plain}  
\newtheorem{theorem}{Theorem}[section] 
\newtheorem{lemma}[theorem]{Lemma} 
\newtheorem{proposition}[theorem]{Proposition}

\theoremstyle{definition} 

\newtheorem{example}[theorem]{Example}

\theoremstyle{remark} 
\newtheorem{remark}{Remark}

\newcommand{\E}{\mathbb{E}}
\newcommand{\PP}{\mathbb{P}}

\newcommand{\toop}{\stackrel{\PP}{\longrightarrow}}
\newcommand{\schw}{\stackrel{d}{\longrightarrow}}
\newcommand{\stab}{\stackrel{st}{\longrightarrow}}

\newcommand{\toas}{\stackrel{\mbox{\tiny a.s.}}{\longrightarrow}}

\newcommand{\bee}{\begin{equation}}
\newcommand{\eee}{\end{equation}}
\newcommand{\beea}{\begin{array}}
\newcommand{\eeea}{\end{array}}

\begin{document}

\begin{frontmatter}

\title{On  U- and V-statistics for discontinuous It\^o semimartingales}
\runtitle{U- and V-statistics for semimartingales}

\begin{aug}
\author{\fnms{Mark} \snm{Podolskij}\ead[label=e1]{mpodolskij@math.au.dk}},
\author{\fnms{Christian} \snm{Schmidt}\ead[label=e2]{cschmidt@math.au.dk}}
\and
\author{\fnms{Mathias} \snm{Vetter}\corref{}\ead[label=e3]{vetterm@mathematik.uni-marburg.de}}

\runauthor{M.~Podolskij et al.}

\affiliation{Aarhus University\thanksmark{m1} and Philipps-Universit\"at Marburg\thanksmark{m2}}

\address{Mark Podolskij and Christian Schmidt\\
Department of Mathematics\\ Aarhus University\\ Ny Munkegade 118\\ 69120 Aarhus, Denmark\\
\printead{e1}
\phantom{E-mail: }\printead*{e2}}

\address{Mathias Vetter\\
Fachbereich Mathematik und Informatik\\ Philipps-Universit\"at Marburg\\  Hans-Meerwein-Stra\ss e 6, 35032 Marburg, Germany\\
\printead{e3}}
\end{aug}

\begin{abstract}
In this paper we examine the asymptotic theory for U-statistics and V-statistics of discontinuous It\^o semimartingales that are observed at high frequency. For different types of kernel functions we show laws of large numbers and associated stable central limit theorems. In most of the cases the limiting process will be conditionally centered Gaussian. The structure of the kernel function determines whether the jump and/or the continuous part of the semimartingale contribute to the limit.   
\end{abstract}



\end{frontmatter}

\section{Introduction} \label{Intro}
U- and V-statistics are classical objects in mathematical statistics. They were introduced in the works of Halmos \cite{Ha}, von Mises \cite{vM} and Hoeffding \cite{Hoe}, who provided (amongst others) the first asymptotic results for the case that the underlying random variables are independent and identically distributed. Since then there was a lot of progress in this field and the results were generalized in various directions. Under weak dependency assumptions asymptotic results are for instance shown in Borovkova et al. \cite{BBD2}, in Denker and Keller \cite{DK2} or more recently in Leucht \cite{AL}. The case of long memory processes is treated in Dehling and Taqqu \cite{DT1,DT2} or in L\'evy-Leduc et al. \cite{LBMTR}. For a general overview we refer to the books of Serfling \cite{Se} and Lee \cite{Le}. The methods applied in the proofs are quite different. One way are decomposition techniques like the famous Hoeffding decomposition or Hermite expansion as for example in Dehling and Taqqu \cite{DT1,DT2} or in L\'evy-Leduc et al. \cite{LBMTR}. Another approach is to use empirical process theory (see e.g. Beutner and Z\"ahle \cite{BZ1} or Podolskij et al. \cite{PSZ}). In Beutner and Z\"ahle \cite{BZ2} this method was recently combined with a continuous mapping approach to give a unifying way to treat the asymptotic theory for both U- and V-statistics in the degenerate and non-degenerate case. 

In this paper we are concerned with U- and V-statistics where the underlying data comes from a (possibly discontinuous) It\^o semimartingale of the form
\begin{align}\label{mint}
X_t=X_0 +\int_{0}^t b_s ds + \int_0^t \sigma_s dW_s + J_t, \quad t\geq 0,
\end{align}
where $W$ is a standard Brownian motion, $(b_s)_{s \geq 0}$ and $(\sigma_s)_{s \geq 0}$ are stochastic processes and $J_t$ is some jump process which will be specified later. Semimartingales play an important role in stochastic analysis because they form a large class of integrators with respect to which the It\^o integral can be defined. This is one reason why they are widely used in applications, for instance in mathematical finance. Since the seminal work of Delbaen and Schachermayer \cite{DS} it is further known that under certain no arbitrage conditions asset price processes must be semimartingales. Those price processes are nowadays observed very frequently, say for example at equidistant time points $0,1/n,\dots ,\left\lfloor nT\right\rfloor/n$ for a fixed $T \in \mathbb{R}$ and large $n$. A solid understanding of the statistical methods based on $X_0, X_{1/n},\dots,X_{\left\lfloor nT\right\rfloor/n}$ is therefore of great interest. In particular, we are interested in the limiting behavior when $n$ tends to infinity. This setting is known as high frequency or infill asymptotics and is an active field of research since the last two decades. For a comprehensive account we refer to the book of Jacod and Protter \cite{JP}.

In Podolskij et al. \cite{PSZ} an asymptotic theory for U-statistics of continuous It\^o semimartingales (i.e. those with $J_t \equiv 0$ in \eqref{mint}) was developed in the high frequency setting, where a U-statistic of order $d$ is defined by
\[
U(X,H)_t^n=\binom{n}{d}^{-1} \sum_{1\leq i_1 <...<i_d \leq \left\lfloor nt\right\rfloor} H(\sqrt{n}\Delta_{i_1}^n X,\dots,\sqrt{n}\Delta_{i_d}^n X),\qquad (\Delta_i^n X=X_{i/n}-X_{(i-1)/n}) 
\]   
for some sufficiently smooth kernel function $H:\mathbb{R}^d \to \mathbb{R}$. The authors have shown that $U(X,H)_t^n$ converges in probability to some functional of the volatility $\sigma$. Also an associated functional central limit theorem was further given, where the limiting process turned out to be conditionally Gaussian. 

In this paper we extend those results to the case of discontinuous It\^o semimartingales $X$. A general problem when dealing with discontinuous processes is that, depending on the function $H$, the U-statistic defined above might not converge to a finite limit at all. Therefore we will deal with slightly different V-statistics of order $d$, given by
\[
Y_t^n(H,X,l)=\frac{1}{n^l} \sum_{\textbf{i}\in \mathcal{B}_t^n (l)}\sum_{\textbf{j}\in \mathcal{B}_t^n (d-l)} H(\sqrt{n} \Delta_{\mathbf{i}}^n X, \Delta_{\mathbf{j}}^n X),
\]
where $0\leq l \leq d$ and
\[
\mathcal{B}_t^n(k) = \left\{\mathbf{i}=(i_1,\dots,i_k) \in \mathbb{N}^{k} | 1 \leq i_1, \dots,i_k \leq \left\lfloor nt\right\rfloor \right\}\quad \quad (k\in\mathbb{N}).
\] 
In the definition of $Y_t^n(H,X,l)$ we used a vector notation, that we will employ throughout the paper: For $\mathbf{s}=(s_1,\dots, s_d) \in \mathbb{R}^d$ and any stochastic process $(Z_s)_{s\in \mathbb{R}}$, we write
\[
Z_{\mathbf{s}}=(Z_{s_1},\dots,Z_{s_d}). 
\]
Comparing the definitions of the U- and V-statistics we see that they are of similar type if $l=d$. In fact, for continuous $X$, both statistics will converge to the same limit if $H$ is symmetric. A major difference is the missing scaling inside the function $H$ whenever $l \neq d$, and this is due to jumps. 

Already the case $d=1$ shows why we need different scalings for different functions $H$. Jacod \cite{J2} (among others) considers the statistics 
\begin{align}\label{glln}
Y_t^n(H,X,1)=\frac{1}{n}\sum_{i=1}^{\left\lfloor nt\right\rfloor} H(\sqrt{n} \Delta_i^n X) \quad \text{and} \quad Y_t^n(H,X,0)=\sum_{i=1}^{\left\lfloor nt\right\rfloor} H(\Delta_i^n X)
\end{align}
for quite general functions $H$, but with a strong view on power variations, i.e.\ $H_p(x)=|x|^p$. For $0<p<2$, and under some mild additional assumptions, Jacod \cite{J2} shows
\begin{align}\label{cest}
Y_t^n(H_p,X,1) \toop m_p \int_0^t |\sigma_s|^p ds,
\end{align}
where $m_p$ is the $p$-th absolute moment of a standard normal distribution. It follows that $Y_t^n(H_p,X,0)$ explodes for this specific $H_p$. On the other hand, if $p>2$ we have
\begin{align}\label{jest}
Y_t^n(H_p,X,0)\toop \sum_{s\leq t} |\Delta X_s|^p,
\end{align}
where $\Delta X_s =\Delta X_s - \Delta X_{s-}$ stands for the jumps of $X$. Clearly, it is now $Y_t^n(H_p,X,1)$ which diverges. 

For the associated central limit theorems the assumptions need to be stronger. Precisely, one requires $0<p<1$ for $Y_t^n(H_p,X,1)$ and $p>3$ for $Y_t^n(H_p,X,0)$. The limiting processes are also (often) conditionally Gaussian, but of different form. For $p<1$ the conditional variance of the limit depends only on the continuous part of $X$, whereas in the case $p>3$ the conditional variance is more complicated and depends on both the jump and the continuous part of $X$. 

To accommodate these different behaviors into our setting, we will consider V-statistics $Y_t^n(H,X,l)$ of order $d$ which are determined by kernel functions of the form
\[
H(x_1,\dots, x_l,y_1, \dots, y_{d-l})=|x_1|^{p_1} \cdot \ldots \cdot |x_l|^{p_l} |y_1|^{q_1}\cdot \ldots \cdot |y_{d-l}|^{q_{d-l}} L(x_1,\dots, x_l,y_1, \dots, y_{d-l}), 
\]
where $L$ has to fulfill some boundedness conditions and needs to be sufficiently smooth. Further we assume $p_1,\dots, p_l <2$ and $q_1,\dots, q_{d-l} >2$. Clearly there are two special cases. If $l=0$ we need a generalization of \eqref{jest} to V-statistics of higher order. If $l=d$ the V-statistic is of similar form as the U-statistic $U(X,H)_t^n$ defined above. In particular, we have to extend the theory of U-statistics of continuous It\^o semimartingales in \cite{PSZ} to the case of discontinuous It\^o semimartingales. Finally, in the sophisticated situation of arbitrary $l$, we will combine the two special cases. The limiting processes in the central limit theorems will still be (in most cases) conditionally Gaussian, with the same structural differences as for the plain power variations. 

The paper is organized as follows. The short section \ref{prel} contains some basic definitions and notations. In section \ref{tjc} we start with the jump case and present a law of large numbers and a central limit theorem in the case $l=0$, but for slightly more general statistics than $Y_t^n(H,X,0)$. A statistical application regarding possible jump sizes is sketched as well. Section \ref{tmc}, on the other hand, is concerned with a law of large numbers and an associated central limit theorem for $Y_t^n(H,X,l)$ and arbitrary $l$. Here, we rely on the previously established results from section \ref{tjc} and on a uniform central limit theorem for U-statistics, which generalizes the results given in Podolskij et al. \cite{PSZ}. Finally, an appendix contains proofs of some technical results, alongside with a proof of the aforementioned uniform central limit theorem for U-statistics.

\section{Preliminaries}\label{prel}
Throughout the paper we assume that we observe a one-dimensional It\^o-semimartingale
\begin{align*}
X_t=X_0 +\int_{0}^t b_s ds + \int_0^t \sigma_s dW_s +(\delta \mathbbm{1}_{\left\{ |\delta| \leq 1\right\}})\ast (\mathfrak{p} - \mathfrak{q})_t +(\delta \mathbbm{1}_{\left\{ |\delta|>1 \right\}})\ast \mathfrak{p}_t, \quad t \in [0,T],
\end{align*}
which is defined on a filtered probability space $(\Omega, \mathcal{F}, (\mathcal{F}_t)_{t \geq 0}, \mathbb{P})$ that satisfies the usual assumptions. Obviously we have $T > 0$, and we require further that $W$ is a Brownian motion and $\mathfrak{p}$ is a Poisson random measure with compensator $\mathfrak{q}(dt,dz)=dt \otimes \lambda(dz)$ for some $\sigma$-finite measure $\lambda$. Unless strengthened, we work with mild assumptions on the coefficients and assume that $b$ is locally bounded, $\sigma$ is c\`adl\`ag and $\delta$ is predictable. Observations come in an equidistant way, i.e.\ we observe $X_0, X_{1/n},\dots,X_{\left\lfloor nT\right\rfloor/n}$, and eventually $n \to \infty$. 

Moreover we will use the following vector notation: If $\mathbf{p}=(p_1,\dots,p_d), \mathbf{x}=(x_1,\dots,x_d) \in \mathbb{R}^d$, then we let $|\mathbf{x}|^{\mathbf{p}}:=\prod_{k=1}^d |x_k|^{p_k}$. Define further $\mathbf{p} \leq \mathbf{x} \Longleftrightarrow p_i \leq x_i$ for all $1 \leq i \leq d$. If $t \in \mathbb{R}$ we let $\mathbf{x} \leq t \Longleftrightarrow x_i \leq t$ for all $1 \leq i \leq d$. By $\left\|\cdot\right\|$ we denote the maximum norm for vectors and the supremum norm for functions. Finally, we introduce the notation
\begin{align}\label{power}
\mathfrak{P}(l):=\left\{p(x_1,\dots,x_l)=\sum_{\mathbf{\alpha} \in A} |x_1|^{\alpha_1}\cdots |x_l|^{\alpha_l} \Big| A \subset \mathbb{R}_{+}^l \text{ finite}  \right\}.
\end{align}
We will assume in the entire paper that $K$ is some generic constant which may change from line to line.

\begin{section}{The jump dominated case}\label{tjc}
In this section we analyze the asymptotic behavior of the V-statistic $V(H,X,l)_t^n$ defined by
\begin{align}\label{defvj}
V(H,X,l)_t^n:= \frac{1}{n^{d-l}} \sum_{\mathbf{i}\in \mathcal{B}_t^n(d)}  H(\Delta_{\mathbf{i}}^n X)=\frac{1}{n^{d-l}} Y_t^n(H,X,0)
\end{align} 
for different types of continuous functions $H:\mathbb{R}^d \to \mathbb{R}$, where the jump part of
$X$ will dominate the limit.  
As a toy example in the case $d=2$ serve the two kernel functions
\[
H_1(x_1,x_2)=|x_1|^p \quad \text{and} \quad  H_2(x_1,x_2)=|x_1 x_2|^p
\]
for some $p >2$. Already for these basic functions it is easy to see why there should be different rates of convergence, i.e.\ different $l$, in the law of large numbers. Consider
\[
V(H_1,X,l)_t^n = \frac{\left\lfloor nt\right\rfloor}{n^{2-l}} \sum_{i=1}^{\left\lfloor nt\right\rfloor}  |\Delta_{i}^n X|^p \quad \text{and} \quad V(H_2,X,l)_t^n = \frac{1}{n^{2-l}} \Bigg(\sum_{i=1}^{\left\lfloor nt\right\rfloor}  |\Delta_{i}^n X|^p\Bigg)^2.
\]
In order to get convergence in probability to some non-trivial limit we know from the 1-dimensional theory (see \eqref{jest}) that we have to choose $l=1$ for $H_1$ and $l=2$ for $H_2$.

In the following two subsections we will provide a law of large numbers and an associated central limit theorem for the statistics defined in \eqref{defvj}.  

\begin{subsection}{Law of large numbers}
For the law of large numbers we do not need to impose any additional assumptions on the process $X$. We only need to require that the kernel function $H$ fulfills \eqref{alln}, which is the same condition as given in \cite{J2} for $d=1$. 
\begin{theorem}\label{lln}
Let $H: \mathbb{R}^d \to \mathbb{R}$ be continuous and $1\leq l \leq d$ such that
\begin{align} \label{alln}
\lim_{(x_1,\dots, x_l) \to \mathbf{0}} \frac{H(x_1,\dots,x_d)}{|x_1|^2 \cdot \ldots \cdot|x_l|^2} =0.
\end{align}
Then, for fixed $t>0$,
\[
V(H,X,l)_t^n \toop V(H,X,l)_t:=t^{d-l} \sum_{\mathbf{s} \in (0,t]^l} H(\Delta X_{\mathbf{s}},\mathbf{0}).
\]
\end{theorem}

\begin{remark} Note that we can write $H$ in the form 
\[
H=|x_1 \cdot \ldots \cdot x_l|^2 L(x_1,\dots,x_d),
\]
where
\[
L(x_1,\dots,x_d)=\begin{cases}
  \frac{H(x_1,\dots,x_d)}{|x_1 \cdot \ldots \cdot x_l|^2},  & \text{if } x_1,\dots,x_l \neq 0,\\
  0, & \text{otherwise }.
\end{cases}
\]
By assumption \eqref{alln}, $L$ is continuous and consequently the limit $V(H,X,l)_t$ is well-defined, since the squared jumps of a semimartingale are absolutely summable. \end{remark}

\begin{remark} Condition \eqref{alln} is stated in a somewhat asymmetric way because it only concerns the first $l$ arguments of $H$. Generally one should rearrange the arguments of $H$ in a way such that \eqref{alln} is fulfilled for the largest possible $l$. In particular, $H(x_1,\dots,x_l,\mathbf{0})$ is not identically 0 then (unless $H \equiv 0$), which will lead to non-trivial limits. 
\end{remark}

\begin{proof}
Let $t>0$ be fixed. The proof will be divided into two parts. In the first one we will show that
\begin{align*}
\xi_t^n:=\frac{1}{n^{d-l}} \sum_{\mathbf{i}\in \mathcal{B}_t^n(d)}  (H(\Delta_{\mathbf{i}}^n X)- H(\Delta_{i_1}^n X,\dots,\Delta_{i_l}^n X, \mathbf{0}))  \toop 0.
\end{align*}
Then we are left with proving the theorem in the case $l=d$, which will be done in the second part.

Since the paths of $X$ are c\`adl\`ag and therefore bounded on compacts by a constant $A_t (\omega)=\sup_{0\leq s \leq t} |X_s(\omega)|$, we have the estimate
\begin{align*}
|\xi_t^n| \leq \frac{1}{n^{d-l}} \sum_{\mathbf{i}\in \mathcal{B}_t^n(d)} |\Delta_{i_1}^n X \cdot \ldots \cdot \Delta_{i_l}^n X|^2 \delta_{L,A_t} (\max(|\Delta_{i_{l+1}}^n X|,\dots,|\Delta_{i_d}^n X|))\\
= \Bigg(\sum_{i=1}^{\left\lfloor nt \right\rfloor} |\Delta_i^n X|^2 \Bigg)^{l} \frac{1}{n^{d-l}}\sum_{i_{l+1},\dots,i_d=1}^{\left\lfloor nt \right\rfloor}\delta_{L,A_t} (\max(|\Delta_{i_{l+1}}^n X|,\dots,|\Delta_{i_d}^n X|)),
\end{align*}
where 
\[
\delta_{L,A_t}(\epsilon):=\sup \left\{|L(\mathbf{x})-L(\mathbf{y})| \Big| \mathbf{x},\mathbf{y} \in [-2A_t,2A_t]^d, \left\|\mathbf{x}-\mathbf{y}\right\| < \epsilon\right\}, \quad \epsilon > 0
\] 
denotes the modulus of continuity of $L$. 

We will now use the elementary property of the c\`adl\`ag process $X$, that for every $\epsilon >0$, there exists $N\in \mathbb{N}$ such that $|\Delta_i^n X| <2\epsilon$ for all $n \geq N$, if $X$ does not have a jump of size bigger than $\epsilon$ on $\big( \frac{i-1}{n}, \frac{i}{n} \big]$. Since the number of those jumps is finite, we obtain for sufficiently large $n$ the estimate
\[
\frac{1}{n^{d-l}}\sum_{i_{l+1},\dots,i_d=1}^{\left\lfloor nt \right\rfloor}\delta_{L,A_t} (\max(|\Delta_{i_{l+1}}^n X|,\dots,|\Delta_{i_d}^n X|)) \leq t^{d-l}\delta_{L,A_t} (2\epsilon) + \frac{K(\epsilon)}{n}.
\] 
Using the continuity of $L$, the left hand side becomes arbitrarily small, if we first choose $\epsilon$ small and then $n$ large. From \cite{JS} we know that
\begin{align} \label{qv}
[X,X]_t^n:=\sum_{i=1}^{\left\lfloor nt \right\rfloor} |\Delta_i^n X|^2 \toop [X,X]_t =\int_0^t \sigma_s^2 \ ds + \sum_{0<s\leq t} |\Delta X_s|^2,
\end{align}
and thus we obtain $\xi_t^n \toop 0$.

For the second part of the proof, i.e.\ the convergence $V(H,X,l)_t^n \toop V(H,X,l)_t$ in the case $l=d$, we define the functions $g_k^n :\mathbb{R}^{d-1} \to \mathbb{R}$ by
\begin{align*}
g_k^n (\mathbf{x})=\sum_{i=1}^{\left\lfloor nt \right\rfloor} |\Delta_i^n X|^2 &L(x_1,\dots,x_{k-1},\Delta_i^n X,x_{k},\dots,x_{d-1}) \\
&- \sum_{s \leq t} |\Delta X_s|^2 L(x_1,\dots,x_{k-1},\Delta X_s,x_{k},\dots,x_{d-1})
\end{align*}
and deduce
\begin{align*}
 &|V(H,X,d)_t^n-V(H,X,d)_t| =\Big|\sum_{\mathbf{i}\in \mathcal{B}_t^n(d)}  H(\Delta_{\mathbf{i}}^n X)- \sum_{\mathbf{s} \in [0,t]^d} H(\Delta X_{\mathbf{s}})\Big| \\
 =& \Big|\sum_{k=1}^d \Big\{\sum_{\mathbf{i}\in \mathcal{B}_t^n(k)} \sum_{\mathbf{s} \in [0,t]^{d-k}}  H(\Delta_{\mathbf{i}}^n X,\Delta X_{\mathbf{s}})- \sum_{\mathbf{i}\in \mathcal{B}_t^n(k-1)}\sum_{\mathbf{s} \in [0,t]^{d-k+1}} H(\Delta_{\mathbf{i}}^n X,\Delta X_{\mathbf{s}}) \Big\} \Big|\\
\leq & \sum_{k=1}^d ([X,X]_t^n)^{k-1} [X,X]_t^{d-k} \sup_{\left\|\mathbf{x}\right\| \leq A_t} |g_k^n(\mathbf{x})|.
\end{align*}
By using \eqref{qv} again we see that it remains to show $\sup_{\left\|\mathbf{x}\right\| \leq A_t} |g_k^n(\mathbf{x})|\toop 0$ for $1\leq k \leq d$. In the following we replace the supremum by a maximum over a finite set and give sufficiently good estimates for the error that we make by doing so. 

For any $m \in \mathbb{N}$ define the (random) finite set $A_t^m $ by
\[
A_t^m := \Big\{ \frac{k}{m} \Big|  \ k \in \mathbb{Z}, \frac{|k|}{m} \leq A_t \Big\}.
\]  
Then we have
\[
\sup_{\left\|\mathbf{x}\right\| \leq A_t} |g_k^n(\mathbf{x})| \leq \max_{\mathbf{x} \in (A_t^m)^{d-1}} |g_k^n(\mathbf{x})| + \sup_{\stackrel{\left\|\mathbf{x}\right\|,\left\|\mathbf{y}\right\| \leq A_t}{\left\|\mathbf{x} - \mathbf{y} \right\|\leq 1/m}}|g_k^n(\mathbf{x})-g_k^n(\mathbf{y})|=:\zeta_{k,1}^n(m) + \zeta_{k,2}^n(m).
\]
Since the sets $A_t^m$ are finite, we immediately get $\zeta_{k,1}^n(m) \toas 0$ as $n \to \infty$ from Remark 3.3.3 in \cite{JP} for any fixed $m$. For the second summand $\zeta_{k,2}^n (m)$ observe that
\[
|\zeta_{k,2}^n (m)| \leq \Big(\sum_{i=1}^{\left\lfloor nt\right\rfloor} |\Delta_i^n X|^2 + \sum_{s \leq t} |\Delta X_s|^2\Big) \delta_{L,A_t}(m^{-1}),
\]   
which implies
\[
\lim_{m\to \infty} \limsup_{n \to \infty} \mathbb{P}(|\zeta_{k,2}^n (m)| >\epsilon) =0 \quad \text{for every} \quad \epsilon >0.
\]
The proof is complete.\end{proof}
\end{subsection}

\begin{subsection}{Central limit theorem}\label{cltjc}
In this section we will show a central limit theorem that is associated to the law of large numbers in Theorem \ref{lln}. The mode of convergence will be the so-called stable convergence. This notion was introduced by Renyi \cite{R} and generalized the concept of weak convergence. We say that a sequence $(Z_n)_{n\in \mathbb{N}}$ of random variables defined on a probability space $(\Omega,\mathcal{F},\mathbb{P})$ with values in a Polish space $(E,\mathcal{E})$ converges stably in law to a random variable $Z$, that is defined on an extension $(\tilde \Omega,\tilde{\mathcal{F}},\tilde{\mathbb{P}})$ of $(\Omega,\mathcal{F},\mathbb{P})$  and takes also values in $(E,\mathcal{E})$, if and only if
\[
\mathbb{E}(f(Z_n)Y) \to \tilde{\mathbb{E}}(f(Z)Y) \quad \text{as} \quad n \to \infty
\]
for all bounded and continuous $f$ and any bounded, $\mathcal{F}$-measurable $Y$. We write $Z_n \stab Z$ for stable convergence of $Z_n$ to $Z$.  For a short summary of the properties of stable convergence we refer to \cite{PV}. The main property that we will use here is that if we have two sequences $(Y_n)_{n \in \mathbb{N}} ,(Z_n)_{n \in \mathbb{N}}$ of real-valued random variables and real-valued random variables $Y,Z$ with $Y_n \toop Y$ and $Z_n \stab Z$, then the joint stable convergence $(Z_n, Y_n) \stab (Z,Y)$ can be concluded.

In contrast to the law of large numbers, we need to impose a mild boundedness assumption on the jumps of the process $X$. We assume that $|\delta(\omega, t, z)|\wedge 1 \leq \Gamma_n(z)$ for all $t \leq \tau_n(\omega)$, where $\tau_n$ is an increasing sequence of stopping times going to infinity. The functions $\Gamma_n$ are assumed to fulfill
\[
\int \Gamma_n^2 \lambda(dz) < \infty.
\]
Since the main result of this section, which is  Theorem \ref{stabvj}, is stable under stopping, we may as well  
assume by a standard localization argument (see \cite[section 4.4.1]{JP}) that all locally bounded processes are in fact bounded, i.e.\
\begin{align*}
|b_t|\leq A, \quad |\sigma_t| \leq A, \quad |X_t| \leq A,\quad |\delta(t,z)| \leq \Gamma(z) \leq A
\end{align*}
holds uniformly in $(\omega,t)$ for some constant $A$ and a function $\Gamma$ with 
\[
\int \Gamma(z)^2 \lambda(dz) \leq A.
\]

A common technique for proving central limit theorems for discontinuous semimartingales is to decompose the process $X$ for fixed $m \in \mathbb{N}$ into the sum of two processes $X(m)$ and $X'(m)$, where the part $X'(m)$ basically describes the jumps of $X$, which are of size bigger than $1/m$ and of whom there are only finitely many. Eventually one lets $m$ go to infinity. 
 
So here we define $D_m = \left\{z: \Gamma(z) >1/m \right\}$ and $(S(m,j))_{j \geq 1}$ to be the successive jump times of the Poisson process $\mathbbm{1}_{\left\{ D_m \backslash D_{m-1}\right\} } \ast \mathfrak{p}$. Let $(S_q)_{q\geq 1}$ be a reordering of $(S(m,j))$, and $$\mathcal{P}_m=\left\{p : S_p = S(k,j)\text{ for } j\geq 1, k\leq m \right\}, \quad \mathcal{P}_t^n(m)=\left\{p\in \mathcal{P}_m : S_p \leq \frac{\left\lfloor nt\right\rfloor}{n} \right\}, \quad \mathcal{P}_t(m)=\left\{p\in \mathcal{P}_m : S_p \leq t\right\}.$$ 
Further let
\begin{align*}
&R_{-}(n,p) = \sqrt{n} (X_{S_p-}-X_{\frac{i-1}{n}})\\
&R_{+}(n,p) = \sqrt{n} (X_{\frac{i}{n}}-X_{S_p}) \\
&R(n,p)=R_{-}(n,p)+R_{+}(n,p),
\end{align*}
if $\frac{i-1}{n}< S_p \leq \frac{i}{n}$.
Now we split $X$ into a sum of $X(m)$ and $X'(m)$, where $X'(m)$ is the "big jump part" and $X(m)$ is the remaining term, by setting
\begin{align*}
&b(m)_t =b_t - \int_{\left\{ D_m \cap \left\{ z: |\delta(t,z)| \leq 1\right\}\right\}} \delta(t,z) \lambda(dz) \\
&X(m)_t=\int_{0}^t b(m)_s ds + \int_0^t \sigma_s dW_s +(\delta \mathbbm{1}_{D_m^c})\ast (\mathfrak{p} - \mathfrak{q})_t \\
&X'(m)=X-X(m)=(\delta \mathbbm{1}_{D_m})\ast \mathfrak{p}.
\end{align*}
Further let $\Omega_n(m)$ denote the set of all $\omega$ such that the intervals $(\frac{i-1}{n},\frac{i}{n}]$  $(1\leq i \leq n)$ contain at most one jump of $X'(m)(\omega)$, and 
\[
|X(m)(\omega)_{t+s}  -X(m)(\omega)_t| \leq \frac{2}{m} \quad \text{for all} \quad t \in [0,T], s\in [0,n^{-1}].
\]
Clearly, $\mathbb{P}(\Omega_n(m)) \to 1$, as $n \to \infty$.

Before we state the main result of this section we begin with some important lemmas. The first one gives useful estimates for the size of the increments of the process $X(m)$. For a proof see \cite[$(2.1.44)$ and $(5.1.24)$]{JP}.  
\begin{lemma}\label{lem1}
For any $p \geq 1$ we have
\[
\mathbb{E}(|X(m)_{t+s}-X(m)_t|^p| \mathcal{F}_{t}) \leq K(s^{(p/2) \wedge 1} +m^p s^p)
\]
for all $t\geq 0, s\in[0,1]$. 
\end{lemma}
As a simple application of the lemma we obtain for $p\geq 2$ and $\mathbf{i} \in \mathcal{B}_t^n(d)$ with $ i_1 <\dots < i_d$
\begin{align*}
& \mathbf{E} \big[ |\Delta_{i_1}^n X(m)|^p \cdot \ldots \cdot |\Delta_{i_d}^n X(m)|^p  \big] =\mathbf{E}\Big[\Delta_{i_1}^n X(m)|^p \cdot \ldots \cdot |\Delta_{i_{d-1}}^n X(m)|^p \mathbf{E} \big[ |\Delta_{i_d}^n X(m)|^p  \big| \mathcal{F}_{\frac{i_d-1}{n}} \big] \Big] \\
\leq & K\Big( \frac{1}{n}  +\frac{m^p}{n^p}\Big)\mathbf{E} \big[ |\Delta_{i_1}^n X(m)|^p \cdot \ldots \cdot |\Delta_{i_{d-1}}^n X(m)|^p  \big] \leq \dots \leq \frac{K(n,m)}{n^d}
\end{align*}
for some positive sequence $K(n,m)$ which satisfies $\limsup_{n\to \infty} K(n,m) \leq K$ for any fixed $m$. Consequently, for general $\mathbf{i} \in \mathcal{B}_t^n(d)$, we have
\[
\mathbf{E} \big[ |\Delta_{i_1}^n X(m)|^p \cdot \ldots \cdot |\Delta_{i_d}^n X(m)|^p  \big] \leq K(n,m) n^{-\#\left\{i_1,\dots,i_d \right\}}.
\]
Since the number of elements $\mathbf{i}=(i_1,\dots,i_d) \in \mathcal{B}_t^n(d)$ with $\#\left\{i_1,\dots,i_d \right\}=k$ is of order $n^k$, we obtain the useful formula
\begin{align}\label{ord1}
\mathbb{E}\Big[\sum_{\mathbf{i} \in \mathcal{B}_t^n(d)} |\Delta_{i_1}^n X(m)|^p \cdot \ldots \cdot |\Delta_{i_d}^n X(m)|^p \Big] \leq K(n,m),
\end{align}
and similarly 
\begin{align}\label{ord2}
\frac{1}{\sqrt{n}}\mathbb{E}\Big[\sum_{\mathbf{i} \in \mathcal{B}_t^n(d)} |\Delta_{i_1}^n X(m)|^p \cdot \ldots \cdot |\Delta_{i_{d-1}}^n X(m)|^p |\Delta_{i_d}^n X(m)| \Big] \leq K(n,m).
\end{align}

The next lemma again gives some estimate for the process $X(m)$ and is central for the proof of Theorem~\ref{stabvj}.  
\begin{lemma}\label{lem2}
Let $C>0$ be a constant. Assume further that $f:\mathbb{R} \times [-C,C]^{d-1} \to \mathbb{R}$ is defined by $f(\mathbf{x})=|x_1|^p g(\mathbf{x})$, where $p>3$ and $g \in \mathcal{C} (\mathbb{R} \times [-C,C]^{d-1})$ is twice continuously differentiable in the first argument. Then we have
\[
\mathbb{E}\Big(\mathbbm{1}_{\Omega_n(m)}\sqrt{n}\Big|\sum_{ i=1 }^{\left\lfloor nt\right\rfloor} \Big(f(\Delta_i^n X(m),x_2,\dots,x_d) - \sum_{\frac{i-1}{n}<s\leq \frac{i}{n}} f(\Delta X(m)_s,x_2,\dots,x_d)\Big)\Big|\Big)  \leq \beta_m(t)
\]
for some sequence $(\beta_m(t))$ with $\beta_m(t) \to 0$ as $m \to \infty$, uniformly in $x_2,\dots,x_d$.
\end{lemma} 
\begin{proof} The main idea is to apply It\^o formula to each of the summands and then estimate the expected value. For fixed $x_2,\dots, x_d$ this was done in \cite[p. 132]{JP}. We remark that their proof essentially relies on the following inequalities: For fixed $\mathbf{z} \in  [-C,C]^{d-1}$ and $|x|\leq 1/m$ ($m \in \mathbb{N}$) there exists $\beta_m(\mathbf{z})$ such that $\beta_m(\mathbf{z}) \to 0$ as $m\to \infty$ and
\begin{align}\label{unin}
|f(x,\mathbf{z})| \leq \beta_m(\mathbf{z}) |x|^3, \ |\partial_1 f(x,\mathbf{z})| \leq \beta_m(\mathbf{z}) |x|^2, \ |\partial_{11}^2 f(x,\mathbf{z})| \leq \beta_m(\mathbf{z}) |x|.
\end{align}     
Further, for $x, y \in \mathbb{R}$, define the functions 
\[
k(x,y,\mathbf{z})=f(x+y,\mathbf{z})-f(x,\mathbf{z})-f(y,\mathbf{z}), \quad g(x,y,\mathbf{z})=k(x,y,\mathbf{z})-\partial_1 f(x, \mathbf{z})y.
\]  
Following \cite{JP} we obtain for $|x|\leq 3/m$ and $|y|\leq 1/m$ that
\begin{align}\label{unin2}
|k(x,y,\mathbf{z})|\leq K \beta_m(\mathbf{z}) |x||y|, \ |g(x,y,\mathbf{z})|\leq K\beta_m(\mathbf{z}) |x| |y|^2.
\end{align}
Since $f$ is twice continuously differentiable in the first argument and $\mathbf{z}$ lies in a compact set, the estimates under \eqref{unin} and \eqref{unin2} hold uniformly in $\mathbf{z}$, i.e.\ we can assume that the sequence $\beta_m(\mathbf{z})$ does not depend on $\mathbf{z}$, and hence the proof in \cite{JP} in combination with the uniform estimates implies the claim. 
\end{proof}

At last we give a lemma that can be seen as a generalization of the fundamental theorem of calculus.

\begin{lemma}\label{lem3}
Consider a function $f \in \mathcal{C}^d(\mathbb{R}^d)$. Then we have
\begin{align*}
f(x)=f(0)+\sum_{k=1}^d \sum_{1 \leq i_1 < \dots < i_k \leq d} \int_0^{x_{i_1}} \cdots \int_0^{x_{i_k}} \partial_{i_k} \cdots \partial_{i_1} f(g_{i_1,\dots,i_k}(s_1,\dots,s_k)) ds_k \dots ds_1,
\end{align*} 
where $g_{i_1,\dots,i_k}:\mathbb{R}^k \to \mathbb{R}^d$ with
\[
(g_{i_1,\dots,i_k}(s_1,\dots,s_k))_j = \begin{cases}
  0,  & \text{if  } \ j \notin \left\{i_1,\dots,i_k \right\}\\
  s_l, & \text{if  } \ j=i_l.
\end{cases} 
\]
\end{lemma}

\begin{proof} First write
\[
f(x)=f(0)+\sum_{k=1}^d \big(f(x_1,\dots,x_k,0,\dots,0)-f(x_1,\dots,x_{k-1},0,\dots,0) \big),
\]
which yields
\[
f(x)=f(0)+\sum_{k=1}^d \int_{0}^{x_k} \partial_k f (x_1,\dots,x_{k-1},t,0,\dots,0) \ dt.
\]
Now we can apply the first step to the function $g_t(x_1,\dots,x_{k-1}):=\partial_k f(x_1,\dots,x_{k-1},t,0,\dots,0)$ in the integral and by doing this step iteratively we finally get the result.
\end{proof}  

We still need some definitions before we can state the central limit theorem (see for comparison \cite[p.126]{JP}). For the definition of the limiting processes we introduce a second probability space $(\Omega',\mathcal{F}', \mathbb{P}')$ equipped with sequences $(\psi_{k+})_{k\geq 1},(\psi_{k-})_{k\geq 1},$ and $(\kappa_k)_{k\geq 1}$ of random variables, where all variables are independent, $\psi_{k\pm} \sim \mathcal{N}(0,1)$, and $\kappa_k \sim U([0,1])$. We then define a very good filtered extension $(\tilde{\Omega}, \tilde{\mathcal{F}}, (\tilde{\mathcal{F}}_t)_{t \geq 0}, \tilde{\mathbb{P}})$ of the original space by  
\[
\tilde{\Omega}=\Omega \times \Omega', \quad \tilde{\mathcal{F}}=\mathcal{F} \otimes \mathcal{F}', \quad \tilde{\mathbb{P}} = \mathbb{P} \otimes \mathbb{P}'.
\]
Let now $(T_k)_{k \geq 1}$ be a weakly exhausting sequence of stopping times for the jumps of $X$. The filtration $\tilde{\mathcal{F}_t}$ is chosen in such a way that it is the smallest filtration containing $\mathcal{F}_t$ and that $\kappa_k$ and $\psi_{k\pm}$ are $\tilde{\mathcal{F}}_{T_k}$-measurable. Further let
\[
R_k =R_{k-}+R_{k+}, \quad \text{with} \quad R_{k-}=\sqrt{\kappa_k} \sigma_{T_k-}\psi_{k-},\quad R_{k+}=\sqrt{1-\kappa_k} \sigma_{T_k}\psi_{k+}.
\]
Also define the sets
\[
\mathcal{A}_l(d):= \left\{ L \in \mathcal{C}^{d+1}(\mathbb{R}^d) \Big| \ \lim_{\mathbf{y} \to 0} \partial_k L(\mathbf{x},\mathbf{y}) =0 \ \text{for all} \ \mathbf{x}\in \mathbb{R}^l,\ k=l+1,\dots,d  \right\}
\]
for $l=1,\dots,d$. 
\begin{remark}  The following properties hold:
\begin{itemize}
\item[(i)]  $\mathcal{A}_l(d) = \mathcal{C}^{d+1}(\mathbb{R}^d)$ for $l=d$.
\item[(ii)] If $f,g \in  \mathcal{A}_l(d)$, then also $f+g, fg \in \mathcal{A}_l(d)$, i.e.\ $\mathcal{A}_l(d)$ is an algebra.
\item[(iii)] Let $f \in \mathcal{C}^{d+1}(\mathbb{R})$ with $f'(0)=0$, then
\[
L(x_1,\dots, x_d)=f(x_1 \cdot \ldots \cdot x_d) \quad \text{and} \quad L(x_1,\dots,x_d)=f(x_1)+\dots+f(x_d)
\]
are elements of $\mathcal{A}_l(d)$ for all $1\leq l \leq d$.
\end{itemize}
\end{remark}
We obtain the following stable limit theorem.
\begin{theorem}\label{stabvj}
Let $1\leq l\leq d$ and $H:\mathbb{R}^d \to \mathbb{R}$ with $H(\mathbf{x})=|x_1|^{p_1} \cdot \ldots\cdot |x_l|^{p_l} L(\mathbf{x})$, where $p_1,\dots,p_l >3$ and $L \in \mathcal{A}_l(d)$.
For $t>0$ it holds that
\begin{align*}
&\sqrt{n} \Big(V(H,X,l)_t^n - V(H,X,l)_t  \Big) \\
&\quad \quad \quad \qquad \qquad \stab U(H,X,l)_t:= t^{d-l} \sum_{k_1,\dots,k_l:T_{k_1},\dots, T_{k_l} \leq t} \sum_{j=1}^l \partial_j H(\Delta X_{T_{k_1} },\dots , \Delta X_{T_{k_l}} ,\mathbf{0})R_{k_j}.
\end{align*}
The limit is $\mathcal{F}$-conditionally centered with variance
\begin{align*}
\mathbb{E}(U(H,X,l)_t^2| \mathcal{F}) = \frac{1}{2} t^{2(d-l)} \sum_{s\leq t} \Big( \sum_{k=1}^l \bar{V}_k (H,X,l,\Delta X_s)\Big)^2  (\sigma_{s -}^2 + \sigma_{s}^2), 
\end{align*}
where
\begin{align}\label{covj}
\bar{V}_k (H,X,l, y)=\sum_{s_1,\dots,s_{k-1},s_{k+1},\dots,s_l \leq t} \partial_k H(\Delta X_{s_1},\dots,\Delta X_{s_{k-1}},y,\Delta X_{s_{k+1}},\dots,\Delta X_{s_l},\mathbf{0}).
\end{align}
Furthermore, the $\mathcal{F}$-conditional law does not depend on the choice of the sequence $(T_k)_{k \in \mathbb{N}}$, and $U(H,X,l)_t$ is $\mathcal{F}$-conditionally Gaussian if $X$ and $\sigma$ do not have common jump times.
\end{theorem}
\begin{remark}
In the case $d=1$ this result can be found in Jacod \cite{J2} (see Theorem 2.11 and Remark 2.14 therein). A functional version of the central limit theorem in the given form does not exist even for $d=1$. For an explanation see Remark 5.1.3 in \cite{JP}. In order to obtain functional results one generally needs to consider the discretized sequence
\[
\sqrt{n} \Big(V(H,X,l)_t^n - V(H,X,l)_{\left\lfloor nt\right\rfloor /n}  \Big).
\]    
In the proof below we would have to show that all approximation steps hold in probability uniformly on compact sets (instead of just in probability), which seems to be out of reach with our methods. What we could show with our approach, though, is that Theorem \ref{stabvj} holds in the finite distribution sense in $t$.  
\end{remark}
\begin{remark}
In the case that the limit is $\mathcal{F}$-conditionally Gaussian we can get a standard central limit theorem by just dividing by the square root of the conditional variance, i.e.
\[
 \frac{\sqrt{n} \big(V(H,X,l)_t^n - V(H,X,l)_t  \big)}{\sqrt{\mathbb{E}(U(H,X,l)_t^2| \mathcal{F})}} \schw \mathcal{N}(0,1).
\]
Since the conditional variance is generally unknown, we might need to consistently estimate it in order to obtain a feasible central limit theorem.  
\end{remark}
 \begin{proof} In the appendix we will show that $U(H,X,l)_t$ is in fact well-defined and fulfills the aforementioned conditional properties. To simplify notations we will give a proof only for symmetric $L$ and $p_1=\dots=p_l=p$ for some $p>3$.  Note that in this case $H$ is symmetric in the first $l$ components, which implies 
\[
\partial_j H(x_1,\dots,x_l,0,\dots,0) =\partial_1 H(x_j, x_2,\dots,x_{j-1},x_1,x_{j+1},\dots,x_l,0,\dots,0).
\] 
Therefore, we have for fixed $j$
\begin{align*}
&\sum_{k_1,\dots,k_l:T_{k_1},\dots, T_{k_l} \leq t}  \partial_k H(\Delta X_{T_{k_1} },\dots , \Delta X_{T_{k_l}} ,\mathbf{0})R_{k_j}\\
=& \sum_{k_1,\dots,k_l:T_{k_1},\dots, T_{k_l} \leq t}  \partial_1 H(\Delta X_{T_{k_j} },\Delta X_{T_{k_2} },\dots,\Delta X_{T_{k_{j-1}} },\Delta X_{T_{k_1} },\Delta X_{T_{k_{j+1}} },\dots , \Delta X_{T_{k_l}} ,\mathbf{0})R_{k_j} \\
=& \sum_{k_1,\dots,k_l:T_{k_1},\dots, T_{k_l} \leq t}  \partial_1 H(\Delta X_{T_{k_1} },\dots , \Delta X_{T_{k_l}} ,\mathbf{0})R_{k_1},
\end{align*}
and thus the limit can be written as
\begin{align*}
 U(H,X,l)_t= lt^{d-l} \sum_{k_1\dots,k_l:T_{k_1},\dots, T_{k_l} \leq t} \partial_1 H(\Delta X_{T_{k_1} },\dots , \Delta X_{T_{k_l}} ,0 \dots,0)R_{k_1}.
\end{align*}
Later we will prove $\sqrt{n}(V(H,X,l)_{\frac{\left\lfloor nt\right\rfloor}{n}}-V(H,X,l)_t) \toop 0$ as $n \to \infty$, so it will be enough to show the discretized version of the central limit theorem, i.e. 
\begin{align}\label{limst1}
\xi_t^n :=\sqrt{n}(V(H,X,l)_t^n- V(H,X,l)_{\frac{\left\lfloor nt\right\rfloor}{n}}) \stab U(H,X,l)_t.
\end{align}
For the proof of this result we will use a lot of decompositions and frequently apply the following claim.
\begin{lemma}\label{lem6}
Let $(Z_n)_{n\in \mathbb{N}}$ be a sequence of random variables, where, for each $m \in \mathbb{N}$, we have a decomposition $Z_n = Z_n(m) + Z_n'(m)$. If there is a sequence $(Z(m))_{m \in \mathbb{N}}$ of random variables and a random variable $Z$ with
\[
Z_n(m) \xrightarrow[n\to \infty]{st} Z(m), \quad Z(m) \xrightarrow[m \to \infty]{\mathbb{P}} Z, \quad \text{and} \quad \lim_{m \to \infty} \limsup_{n\to \infty} \mathbb{P}(|Z_n'(m)| >\eta)=0 \quad \text{for all} \quad \eta >0,
\]
then
\[
Z_n \stab Z.
\]
\end{lemma}
\noindent For a proof of this result see \cite[Prop. 2.2.4]{JP}. 

For the proof of \eqref{limst1} we will successively split $\xi_t^n$ into several terms and then apply Lemma \ref{lem6}. As a first decomposition we use
\[
\xi_t^n = \mathbbm{1}_{\Omega_n(m)} \xi_t^n + \mathbbm{1}_{\Omega \backslash \Omega_n(m)} \xi_t^n.
\]
Since $\mathbb{P}(\Omega_n(m)) \to 1$ as $n\to \infty$,  the latter term converges to $0$ almost surely as $n \to \infty$, so we can focus on the first summand, which we further decompose into

\begin{align}\label{dcom}
\mathbbm{1}_{\Omega_n(m)} \xi_t^n=\mathbbm{1}_{\Omega_n(m)}\Big(\zeta^n(m) + \sum_{k=0}^l \sum_{j=0}^{d-l} \big(\zeta_{k,j}^n(m) - \tilde{\zeta}_{k,j}^n(m) \big)-\sum_{k=1}^l \zeta_k^n(m)\Big)
\end{align}
with
\begin{align*}
\zeta^n(m)&= \sqrt{n}\Big( \frac{1}{n^{d-l}}\sum_{\textbf{i}\in \mathcal{B}_t^n (d)} H(\Delta_{\textbf{i}}^n X(m)) - \frac{\left\lfloor nt\right\rfloor}{n^{d-l}}^{d-l}\sum_{u_1,\dots,u_l \leq \frac{\left\lfloor nt\right\rfloor}{n}} H(\Delta X(m)_{u_1},\dots,\Delta X(m)_{u_l},\mathbf{0})\Big)  \\
\zeta_{k,j}^n(m)&= \frac{\sqrt{n}}{n^{d-l}}  \sum_{\textbf{p},\textbf{q} \in \mathcal{P}_t^n(m)^{k \times j}} \sideset{}{'}\sum_{\textbf{i} \in \mathcal{B}_t^n (l-k) \atop \textbf{r} \in \mathcal{B}_t^n (d-l-j)} \binom{l}{k} \binom{d-l}{j}  H\Big(\Delta X_{S_{\textbf{p}}}+\frac{R(n,\textbf{p})}{\sqrt{n}} , \Delta_{\textbf{i}}^n X(m),\Delta X_{S_{\textbf{q}}}+\frac{R(n,\textbf{q})}{\sqrt{n}} , \Delta_{\mathbf{r}}^n X(m)\Big)  \\
\tilde{\zeta}_{k,j}^n(m) &= \frac{\sqrt{n}}{n^{d-l}}  \sum_{\textbf{p},\textbf{q} \in \mathcal{P}_t^n(m)^{k \times j}} \sideset{}{'}\sum_{\textbf{i} \in \mathcal{B}_t^n (l-k) \atop \textbf{r} \in \mathcal{B}_t^n (d-l-j)} \binom{l}{k} \binom{d-l}{j}  H\Big(\frac{1}{\sqrt{n}} R(n,\textbf{p}), \Delta_{\textbf{i}}^n X(m),\frac{1}{\sqrt{n}} R(n,\textbf{q}), \Delta_{\textbf{r}}^n X(m) \Big)  \\
\zeta_k^n(m) &= \sqrt{n}\frac{\left\lfloor nt\right\rfloor}{n^{d-l}}^{d-l} \sum_{\textbf{p}\in \mathcal{P}_t^n(m)^k} \sum_{u_{k+1},\dots, u_l \leq \frac{\left\lfloor nt\right\rfloor}{n}} \binom{l}{k} H\Big(\Delta X_{S_{\textbf{p}}} , \Delta X_{u_{k+1}}(m),\dots,\Delta X_{u_l}(m),\mathbf{0}\Big).
\end{align*}
The prime on the sums indicates that we sum only over those indices $\mathbf{i}$ and $\mathbf{r}$ such that $\Delta_{\mathbf{i}}^n X'(m)$ and $\Delta_{\mathbf{r}}^n X'(m)$ are vanishing, which in other word means that no big jumps of $X$ occur in the corresponding time intervals.  
 
The basic idea behind the decomposition is that we distinguish between intervals $(\frac{i-1}{n},\frac{i}{n}]$ where $X$  has a big jump and where not. Essentially we replace the original statistic $\xi_t^n$ by the same statistic $\zeta^n(m)$ for the process $X(m)$ instead of $X$. Using the trivial identity
\[
\sum_{\textbf{i}\in \mathcal{B}_t^n (d)} H(\Delta_{\textbf{i}}^n X)=\sum_{\textbf{i}\in \mathcal{B}_t^n (d)} H(\Delta_{\textbf{i}}^n X(m))+\sum_{\textbf{i}\in \mathcal{B}_t^n (d)} \Big( H(\Delta_{\textbf{i}}^n X)- H(\Delta_{\textbf{i}}^n X(m)) \Big)
\]
we can see that an error term appears by doing this. Of course, we have $\Delta_{\textbf{i}}^n X(m)=\Delta_{\textbf{i}}^n X$ if no big jump occurs. In the decomposition above, $\zeta_{k,j}^n(m)-\tilde{\zeta}_{k,j}^n(m)$ gives the error term if we have $k$ big jumps in the first $l$ coordinates and $j$ big jumps in the last $d-l$ coordinates. In the same manner the term $\zeta_k^n(m)$ takes into account that we might have big jumps in $k$ arguments of $H(\Delta X_{u_1},\dots,\Delta X_{u_l},\mathbf{0})$. All the binomial coefficients appear because of the symmetry of $H$ in the first $l$ and the last $d-l$ arguments. Note also that this decomposition is not valid without the indicator function $\mathbbm{1}_{\Omega_n(m)}$. 

In the appendix we will prove the following claim.
\begin{proposition}\label{ap}
It holds that
\begin{align*} 
\lim_{m \to \infty} \limsup_{n\to \infty} \mathbb{P}\Bigg(\mathbbm{1}_{\Omega_n(m)}\Big|\sum_{k=0}^l \sum_{j=0}^{d-l} \big(\zeta_{k,j}^n(m) - \tilde{\zeta}_{k,j}^n(m)\big) -\sum_{k=1}^l \zeta_k^n(m)-(\zeta_{l,0}^n(m)-\zeta_l^n(m))\Big| >\eta \Bigg)=0
\end{align*}
for all $\eta>0$.
\end{proposition}
So in view of Lemma \ref{lem6} we are left with considering the terms $\zeta_{l,0}^n (m)- \zeta_l^n(m)$ and $\zeta_n(m)$, where the first one is the only one that contributes to the limiting distribution. We will start with proving the three assertions
\begin{align}\label{as1}
\lim_{m \to \infty} \limsup_{n\to \infty} \mathbb{P}(\mathbbm{1}_{\Omega_n(m)}|\zeta_{l,0}^n(m)-\hat{\zeta}_{l,0}^n(m)| >\eta)=0 \quad \text{for all} \quad \eta >0,
\end{align}
\begin{align}\label{as2}
\mathbbm{1}_{\Omega_n(m)}(\hat{\zeta}_{l,0}^n(m)-\zeta_l^n(m)) \stab U(H,X'(m),l)_t, \quad \text{as } n \to \infty,
\end{align}
\begin{align}\label{as3}
U(H,X'(m))_t \stackrel{\tilde{\mathbb{P}}}{\longrightarrow} U(H,X,l)_t, \quad \text{as } m \to \infty,
\end{align}
 where
\begin{align*}
\hat{\zeta}_{l,0}^n(m):= \frac{\sqrt{n}}{n^{d-l}}  \sum_{\textbf{p} \in \mathcal{P}_t^n(m)^{l}} \sum_{\textbf{j} \in \mathcal{B}_t^n (d-l)}  H\Big(\Delta X_{S_{\textbf{p}}}+\frac{R(n,\textbf{p})}{\sqrt{n}} ,\mathbf{0}\Big).
\end{align*}
For \eqref{as1} observe that we have
\begin{align*}
&\mathbbm{1}_{\Omega_n(m)}|\zeta_{l,0}^n(m)-\hat{\zeta}_{l,0}^n(m)|\\
 \leq & \mathbbm{1}_{\Omega_n(m)}  \sum_{\textbf{p} \in \mathcal{P}_t(m)^{l}}  \Big|\Delta X_{S_{\mathbf{p}}} + \frac{1}{\sqrt{n}}R(n,\mathbf{p})  \Big|^p \frac{\sqrt{n}}{n^{d-l}}\sum_{\textbf{j} \in \mathcal{B}_t^n (d-l)}\sum_{k=1}^{d-l}\sup_{\substack{\mathbf{x}\in [-2A,2A]^l \\ \mathbf{y} \in [-2/m,2/m]^{d-l}}}  |\partial_k L(\mathbf{x},\mathbf{y})| |\Delta_{j_{k}}^n X(m)|+O_{\mathbb{P}}(n^{-1/2})
\end{align*}
by the mean value theorem. The error of small order in the estimate above is due to the finitely many large jumps, which are included in the sum over $\textbf{j}$ now, but do not appear in $\zeta_{l,0}^n(m)$ by definition. Clearly,
\[
\lim_{M \to \infty} \limsup_{m\to \infty} \limsup_{n\to\infty} \mathbb{P} \Big(  \sum_{\textbf{p} \in \mathcal{P}_t(m)^{l}}  \Big|\Delta X_{S_{\mathbf{p}}} + \frac{1}{\sqrt{n}}R(n,\mathbf{p})  \Big|^p > M\Big) =0 ,
\]
and by Lemma \ref{lem1} we have
\begin{align*}
&\mathbb{E}\Big( \frac{\sqrt{n}}{n^{d-l}}\sum_{\textbf{j} \in \mathcal{B}_t^n (d-l)}\sum_{k=1}^{d-l}\sup_{\mathbf{x}\in [-2A,2A]^l} \sup_{\mathbf{y} \in [-2/m,2/m]^{d-l}} |\partial_k L(\mathbf{x},\mathbf{y})| |\Delta_{j_{k}}^n X(m)|\Big)\\
 \leq & K (1+mn^{-1/2}) \sup_{\mathbf{x}\in [-2A,2A]^l} \sup_{\mathbf{y} \in [-2/m,2/m]^{d-l}} |\partial_k L(\mathbf{x},\mathbf{y})|,
\end{align*}
which converges to $0$ if we first let $n \to \infty$ and then $m \to \infty$, since $L \in \mathcal{A}_l(d)$ and $[-2A,2A]^l$ is compact. This immediately implies \eqref{as1}.

For the proof of \eqref{as2} we need another Lemma, which can be found in \cite[Prop. 4.4.10]{JP}.
\begin{lemma}\label{lem4}
 For fixed $p \in \mathbb{N}$ the sequence $(R(n,p))_{n\in \mathbb{N}}$ is bounded in probability, and
\[
(R(n,p)_{-},R(n,p)_{+})_{p \geq 1} \stab (R_{p-},R_{p+})_{p \geq 1}
\]
as $n \to \infty$.
\end{lemma}
Then we have, by the mean value theorem, Lemma \ref{lem4}, the properties of stable convergence, and the symmetry of $H$ in the first $l$ components
\begin{align*}
&\mathbbm{1}_{\Omega_n(m)}(\hat{\zeta}_{l,0}^n(m)-\zeta_l^n(m)) \\
=&\sqrt{n}\mathbbm{1}_{\Omega_n(m)}\Bigg( \frac{\left\lfloor nt\right\rfloor^{d-l}}{n^{d-l}}  \sum_{\textbf{p} \in \mathcal{P}_t^n(m)^{l}}   \Big[ H\Big(\Delta X_{S_{\textbf{p}}}+\frac{1}{\sqrt{n}} R(n,\textbf{p}),\mathbf{0}\Big)-   H\Big(\Delta X_{S_{\textbf{p}}} ,\mathbf{0}\Big)\Big] \Bigg) \\
& \quad \quad \quad\stab U(H,X'(m),l)_t=l t^{d-l} \sum_{\textbf{p} \in \mathcal{P}_t(m)^{l}}\partial_1 H\Big(\Delta X_{S_{\textbf{p}}},\mathbf{0}\Big) R_{p_1} \quad \text{as} \quad n\to \infty, 
\end{align*}
i.e.\ \eqref{as2}. For the proof of \eqref{as3} we introduce the notation $\mathcal{P}_t=\left\{p\in \mathbb{N} | S_p \leq t \right\}$. We then use the decomposition 
\begin{align*}
U (H,X,l)_t - U(H,X'(m),l)_t& = lt^{d-l}\sum_{k=1}^l \sum_{\mathbf{p} \in \mathcal{P}_t^{k-1}} \sum_{p_k \in \mathcal{P}_t \backslash \mathcal{P}_t(m)} \sum_{\mathbf{r} \in \mathcal{P}_t(m)^{l-k}} \partial_1 H(\Delta X_{S_{\mathbf{p}}},\Delta X_{S_{p_k}},\Delta X_{S_{\mathbf{r}}},\mathbf{0}) R_{p_1} \\
&=: lt^{d-l}\sum_{k=1}^l \psi_k (m).
\end{align*}
We have to show that, for each $k$, $\psi_k(m)$ converges in probability to $0$ as $m \to \infty$. We will give a proof only for the case $k=1$. Therefore, define the set
\[
A(M):=\left\{\omega \in \Omega \Bigg| \sum_{s \leq t} (|\Delta X_s(\omega)|^p+|\Delta X_s(\omega)|^{2p}+|\Delta X_s(\omega)|^{2p-2}) \leq M\right\}, \quad M \in \mathbb{R}_{+}. 
\] 
Then we have
\begin{align}\label{eq}
\tilde{\mathbb{P}}(|\psi_1(m)| > \eta) \leq \tilde{\mathbb{P}}(|\psi_1(m)| \mathbbm{1}_{A(M)} > \eta/2)+\mathbb{P}(\Omega \backslash A(M)).
\end{align}
By the continuity of $L$ and $\partial_1 L$, and since the jumps of $X$ are uniformly bounded in $\omega$, we get
\begin{align*}
&\tilde{\mathbb{P}}(|\psi_1(m)| \mathbbm{1}_{A(M)} > \eta/2) 
 \leq K \mathbb{E}(  \mathbbm{1}_{A(M)} \tilde{\mathbb{E}}( \psi_1(m)^2 | \mathcal{F}   ) )  \\
\leq& K \mathbb{E}\bigg(\mathbbm{1}_{A(M)}  \sum_{q \in \mathcal{P}_t \backslash \mathcal{P}_t(m)}\bigg(  \sum_{\mathbf{r} \in \mathcal{P}_t(m)^{l-1}} \partial_1 H(\Delta X_{S_q},\Delta X_{S_{\mathbf{r}}},0,\dots,0)\bigg)^2 \bigg) \\
\leq& K \mathbb{E}\bigg(\mathbbm{1}_{A(M)}  \sum_{q \in \mathcal{P}_t \backslash \mathcal{P}_t(m)}(|\Delta X_{S_q}|^p +|\Delta X_{S_q}|^{p-1})^2 \bigg(  \sum_{r \in \mathcal{P}_t(m)} |\Delta X_{S_r}|^p \bigg)^{2(l-1)} \bigg)\\
\leq& K M^{2(l-1)} \mathbb{E}\bigg(\mathbbm{1}_{A(M)}  \sum_{q \in \mathcal{P}_t \backslash \mathcal{P}_t(m)}(|\Delta X_{S_q}|^{2p} +|\Delta X_{S_q}|^{2p-2})\bigg) \to 0 \quad \text{as} \quad m \to \infty
\end{align*}
by the dominated convergence theorem. Since the second summand in \eqref{eq} is independent of $m$ and converges to $0$ as $M \to \infty$, we have
\[
\tilde{\mathbb{P}}(|\psi_1(m)| > \eta) \to 0 \quad \text{for all} \quad \eta >0.
\]
The proof for the convergence in probability of $\psi_k(m)$ to $0$ for $2\leq k\leq l$ is similar. 

It remains to show that
\begin{align}\label{as4}
\lim_{m \to \infty} \limsup_{n \to \infty} \mathbb{P} (\mathbbm{1}_{\Omega_n (m)} |\zeta^n(m)| >\eta) =0 
\end{align}
for all $\eta >0$. 

Again, we need several decompositions. We have
 \begin{align*}
\zeta^n (m) 
=&\sqrt{n}\Big( \frac{1}{n^{d-l}}\sum_{\textbf{i}\in \mathcal{B}_t^n (d)} H(\Delta_{\textbf{i}}^n X(m)) - \frac{\left\lfloor nt\right\rfloor}{n^{d-l}}^{d-l}\sum_{u_1,\dots,u_l \leq \frac{\left\lfloor nt\right\rfloor}{n}} H(\Delta X(m)_{u_1},\dots,\Delta X(m)_{u_l},\mathbf{0})\Big)\\
=& \sqrt{n}\Big( \frac{1}{n^{d-l}}\sum_{\textbf{i}\in \mathcal{B}_t^n (d)} H(\Delta_{\textbf{i}}^n X(m)) - \frac{\left\lfloor nt\right\rfloor}{n^{d-l}}^{d-l}\sum_{\textbf{i}\in \mathcal{B}_t^n (l)} H(\Delta_{\textbf{i}}^n X(m),\mathbf{0})\Big)\\
&+\sqrt{n}\Big( \frac{\left\lfloor nt\right\rfloor}{n^{d-l}}^{d-l}\sum_{\textbf{i}\in \mathcal{B}_t^n (l)} H(\Delta_{\textbf{i}}^n X(m),\mathbf{0}) - \frac{\left\lfloor nt\right\rfloor}{n^{d-l}}^{d-l}\sum_{u_1,\dots,u_l \leq \frac{\left\lfloor nt\right\rfloor}{n}} H(\Delta X(m)_{u_1},\dots,\Delta X(m)_{u_l},\mathbf{0})\Big) \\
=&:\Psi_1^n(m) + \Psi_2^n(m).
\end{align*}
First observe that we obtain by the mean value theorem, and since $X$ is bounded,
\begin{align*}
&\mathbbm{1}_{\Omega_n (m)}|\Psi_1^n(m)| \\
=&\frac{\sqrt{n}}{n^{d-l}}\mathbbm{1}_{\Omega_n (m)} \sum_{\textbf{i}\in \mathcal{B}_t^n (d)} |\Delta_{i_1}^n X(m) \cdots \Delta_{i_l}^n X(m)|^p|L(\Delta_{\mathbf{i}}^n X(m))-L(\Delta_{i_1}^n X(m),\dots,\Delta_{i_l}^n X(m),\mathbf{0})| \\
\leq &  K \mathbbm{1}_{\Omega_n (m)}\frac{\sqrt{n}}{n^{d-l}} \sum_{\textbf{i}\in \mathcal{B}_t^n (d)} \sum_{k=l+1}^d |\Delta_{i_1}^n X(m) \cdots \Delta_{i_l}^n X(m)|^p |\Delta_{i_{k}}^n X(m)| \\
= & K (d-l)\mathbbm{1}_{\Omega_n (m)}\frac{\sqrt{n}}{n^{d-l}} \sum_{\textbf{i}\in \mathcal{B}_t^n (d)} |\Delta_{i_1}^n X(m) \cdots \Delta_{i_l}^n X(m)|^p |\Delta_{i_{l+1}}^n X(m)| \\
\leq & \frac{K(d-l)}{m^{(p-2)l}} \frac{1}{\sqrt{n}} \mathbbm{1}_{\Omega_n (m)} \sum_{\textbf{i}\in \mathcal{B}_t^n (l+1)}|\Delta_{i_1}^n X(m) \cdots \Delta_{i_l}^n X(m)|^2 |\Delta_{i_{l+1}}^n X(m)|.
\end{align*}
By \eqref{ord2} and $\limsup_{n \to \infty} K(m,n) \leq K$ we get
\[
\lim_{m\to \infty}\limsup_{n \to \infty}\mathbb{E}(\mathbbm{1}_{\Omega_n (m)} |\Psi_1^n(m)| )  = 0.
\]

\noindent When showing that $\Psi_2^n(m)$ converges to $0$ we can obviously restrict ourselves to the case $l=d$. We need further decompositions:
\begin{align*}
\Psi_2^n(m)&=\sqrt{n}\sum_{k=1}^d \Big(\sum_{\textbf{i}\in \mathcal{B}_t^n (k)}\sum_{\textbf{s}\in (0,\frac{\left\lfloor nt\right\rfloor}{n}]^{d-k}} H(\Delta_{\textbf{i}}^n X(m),\Delta X(m)_{\textbf{s}}) -\sum_{\textbf{i}\in \mathcal{B}_t^n (k-1)}\sum_{\textbf{s}\in (0,\frac{\left\lfloor nt\right\rfloor}{n}]^{d-k+1}} H(\Delta_{\textbf{i}}^n X(m),\Delta X(m)_{\mathbf{s}}) \Big) \\
&=:\sum_{k=1}^d \Psi_2^n(m,k).
\end{align*}
For a fixed $k$ we have
\begin{align*}
&\Psi_2^n(m,k)= \sum_{\textbf{i}\in \mathcal{B}_t^n (k-1)} |\Delta_{i_1}^n X(m) \cdots \Delta_{i_{k-1}}^n X(m)|^p \sum_{\mathbf{s} \in (0,\frac{\left\lfloor nt\right\rfloor}{n}]^{d-k} } |\Delta X(m)_{s_1} \cdots \Delta X(m)_{s_{d-k}}|^p \\
\times & \sqrt{n} \Big( \sum_{j=1}^{\left\lfloor nt\right\rfloor} |\Delta_j^n X(m)|^p L(\Delta_{\mathbf{i}}^n X(m),\Delta_j^n X(m),\Delta X(m)_{\mathbf{s}}) - \sum_{u \leq \frac{\left\lfloor nt\right\rfloor}{n}} |\Delta X(m)_u|^p L(\Delta_{\mathbf{i}}^n X(m),\Delta X(m)_u,\Delta X(m)_{\mathbf{s}})\Big),
\end{align*}
where we denote the term in the second line by $\Theta_k^n(m,\mathbf{i},\mathbf{s})$. What causes problems here is that $\Theta_k^n(m,\mathbf{i},\mathbf{s})$ depends on the random variables $\Delta_{\mathbf{i}}^n X(m)$ and $\Delta X(m)_{\mathbf{s}}$ and we therefore cannot directly apply Lemma \ref{lem2}. To overcome this problem we introduce the function $f_y \in \mathcal{C}^{d+1}(\mathbb{R}^{d-1})$ defined by
\[
f_y (\mathbf{x}) = |y|^p L(x_1,\dots,x_{k-1},y,x_{k+1},\dots,x_d).
\]
Then we have
\[
\Theta_k^n(m,\mathbf{i},\mathbf{s})=\sqrt{n} \Big( \sum_{j=1}^{\left\lfloor nt\right\rfloor} f_{\Delta_j^n X(m)}(\Delta_{\mathbf{i}}^n X(m),\Delta X(m)_{\mathbf{s}}) - \sum_{u \leq \frac{\left\lfloor nt\right\rfloor}{n}} f_{\Delta X(m)_u} (\Delta_{\mathbf{i}}^n X(m),\Delta X(m)_{\mathbf{s}})\Big).
\]
Now we replace the function $f_y$ according to Lemma \ref{lem3} by
\[
f_y(\mathbf{x})=f_y(\mathbf{0})+\sum_{k=1}^d \sum_{1 \leq i_1 < \dots < i_k \leq d} \int_0^{x_{i_1}} \cdots \int_0^{x_{i_k}} \partial_{i_k} \cdots \partial_{i_1} f_y(g_{i_1,\dots,i_k}(s_1,\dots,s_k)) ds_k \dots ds_1.
\]
Since all of the appearing terms have the same structure we will exemplarily treat one of them:
\begin{align*}
&\sqrt{n} \Big| \sum_{j=1}^{\left\lfloor nt\right\rfloor} \int_0^{\Delta X_{i_1}^n(m)}|\Delta_j^n X(m)|^p \partial_1 L(s_1,0,\dots,0,\Delta_j^n X(m),0,\dots,0) ds_1\\
  & \quad \quad \quad \quad \quad \quad - \sum_{u \leq \frac{\left\lfloor nt\right\rfloor}{n}}  \int_0^{\Delta X_{i_1}^n(m)} |\Delta X(m)_u|^p \partial_1L(s_1,0,\dots,0,\Delta X(m)_u,0,\dots,0) ds_1 \Big| \\
	&\leq \int_{-\frac{2}{m}}^{\frac{2}{m}}\sqrt{n}\Big| \sum_{j=1}^{\left\lfloor nt\right\rfloor} |\Delta_j^n X(m)|^p \partial_1 L(s_1,0,\dots,0,\Delta_j^n X(m),0,\dots,0) \\
	&  \quad \quad \quad \quad \quad \quad -\sum_{u \leq \frac{\left\lfloor nt\right\rfloor}{n}}  |\Delta X(m)_u|^p \partial_1L(s_1,0,\dots,0,\Delta X(m)_u,0,\dots,0) \Big| ds_1.
\end{align*}
This means that we can bound $|\Theta_k^n(m,\mathbf{i},\mathbf{s})|$ from above by some random variable $\tilde{\Theta}_k^n(m)$ which is independent of $\mathbf{i}$ and $\mathbf{s}$ and which fulfills 
\begin{align}\label{neg}
\lim_{m \to \infty} \limsup_{n\to \infty} \mathbb{E} \big[\mathbbm{1}_{\Omega_n(m)}  \tilde{\Theta}_k^n(m) \big] =0
\end{align}
by Lemma \ref{lem2}. Using the previous estimates we have
\[
|\Psi_2^n(m,k)| \leq \tilde{\Theta}_k^n(m) \Bigg(\sum_{j=1}^{\left\lfloor nt\right\rfloor} |\Delta_{j}^n X(m)|^p \Bigg)^{k-1} \Bigg(\sum_{u \leq \frac{\left\lfloor nt\right\rfloor}{n} } |\Delta X(m)_u|^p \Bigg)^{d-k}.
\]
Clearly the latter two terms are bounded in probability and therefore \eqref{neg} yields
\[
\lim_{m \to \infty} \limsup_{n \to \infty} \mathbb{P}(\mathbbm{1}_{\Omega_n(m)} |\Psi_2^n(m)| > \eta) =0,
\]
which proves \eqref{as4}.

The last thing we have to show is  
\begin{align}
\sqrt{n} \Big(V(H,X,l)_t - V(H,X,l)_{\frac{\left\lfloor nt\right\rfloor}{n}} \Big) \toop 0,
\end{align}
e.g.\ in the case $l=d$. From \cite[p. 133]{JP} we know that in the case $d=1$ we have
\begin{align}\label{rest}
\sqrt{n}\sum_{\frac{\left\lfloor nt\right\rfloor}{n} < s_k \leq t} |\Delta X_{s_k}|^p \toop 0.
\end{align}
The general case follows by using the decomposition
\begin{align*}
&\Big|\sqrt{n} \Big(\sum_{s_1,\dots,s_d \leq t} H(\Delta X_{s_1},\dots,\Delta X_{s_d})    - \sum_{s_1,\dots,s_d \leq \frac{\left\lfloor nt\right\rfloor}{n}} H(\Delta X_{s_1},\dots,\Delta X_{s_d})   \Big)\Big| \\
=&\Big|\sqrt{n}\sum_{k=1}^d \Bigg(\sum_{s_1,\dots,s_{k-1}\leq t} \sum_{s_{k+1},\dots,s_{d}\leq \frac{\left\lfloor nt\right\rfloor}{n}} \sum_{\frac{\left\lfloor nt\right\rfloor}{n} < s_k \leq t}H(\Delta X_{s_1},\dots,\Delta X_{s_d}) \Bigg) \Big|\\
\leq &  \sum_{k=1}^d \sum_{s_1,\dots,s_{k-1}\leq t} \sum_{s_{k+1},\dots,s_{d}\leq \frac{\left\lfloor nt\right\rfloor}{n}} |\Delta X_{s_1} \cdots \Delta X_{s_{k-1}} \Delta X_{s_{k+1}} \cdots \Delta X_{s_d}|^p\Big( \sqrt{n}\sum_{\frac{\left\lfloor nt\right\rfloor}{n} < s_k \leq t} |\Delta X_{s_k}|^p \Big) \toop 0.\end{align*} 
Hence the proof of Theorem \ref{stabvj} is complete. 
\end{proof}

As a possible application of the theory let us indicate how one could obtain information about the jump sizes of the process X. For instance, we will sketch a procedure in order to decide whether all sizes lie on a grid $\alpha + \beta \mathbb{Z}$ for a given $\beta$, but unknown $\alpha$.

We start the discussion with a slightly more general situation and consider sets $M\subset \mathbb{R}$, for which we can find a non-negative function $g_M :\mathbb{R} \to \mathbb{R}$ that fulfills $g_M(x)=0$ if and only if $x \in M$, and such that the function $L_M:\mathbb{R}^2 \to\mathbb{R}$ defined by $L_M(x,y)=g_M(x-y)$ lies in $\mathcal{A}_2(2)$. Then our theory shows for $H_M =|x|^{p_1} |y|^{p_2}L_M(x,y)$ that the limit $V(H_M, X, 2)$ vanishes if and only if there is $\alpha \in \mathbb{R}$ such that all (non-zero) jump sizes lie in the set $\alpha + M$. In other words, our theory enables us to construct a test whether such an $\alpha$ exists. As a more explicit example we consider the following one.      

\begin{example} \label{ex1}
For a given $\beta \in \mathbb{R}$ consider the function $H(x,y)=|x|^4|y|^4 \sin^2 \Big(\frac{\pi(x-y)}{\beta} \Big)$. Then we have
\[
\sum_{i,j=1}^{\left\lfloor nt\right\rfloor} H(\Delta_i^n X, \Delta_j^n X) \toop L(\beta):=\sum_{s_1,s_2 \leq t} |\Delta X_{s_1}|^4|\Delta X_{s_2}|^4 \sin^2\Big(\frac{\pi(\Delta X_{s_1}-\Delta X_{s_2})}{\beta} \Big).
\]
It holds that $L(\beta)=0$ if and only if there exists an $\alpha \in \mathbb{R}$ such that 
\[
\Delta X_s \in \alpha + \beta \mathbb{Z} \quad \text{for all} \ \ s\leq t \ \ \text{with} \ \ \Delta X_s \neq 0. 
\]
\end{example}
To formally test whether there exists an $\alpha\in \mathbb{R}$ such that all jump sizes lie in the set $\alpha + \beta \mathbb{Z}$ one would of course need to derive estimators for the conditional variance of the limit in Theorem \ref{stabvj}.
\end{subsection} 

\end{section}

\begin{section}{The mixed case}\label{tmc}
In this section we will present an asymptotic theory for statistics of the form
\begin{align}\label{mst}
Y_t^n(H,X,l)=\frac{1}{n^l} \sum_{\textbf{i}\in \mathcal{B}_t^n (l)}\sum_{\textbf{j}\in \mathcal{B}_t^n (d-l)} H(\sqrt{n} \Delta_{\mathbf{i}}^n X, \Delta_{\mathbf{j}}^n X),
\end{align}
where $H$ behaves like $|x_1|^p \cdots |x_l|^p$ for $p<2$ in the first $l$ arguments and like $|x_{l+1}|^q \cdots |x_{d}|^q$ for $q>2$ in the last $d-l$ arguments. As already mentioned in the introduction, powers smaller than two and powers larger than two lead to completely different limits. This makes the treatment of $Y_t^n(H,X,l)$ for general $l$ way more complicated than in section \ref{tjc} where only large powers appear. In fact, we use the results from section \ref{tjc} and combine them with quite general results concerning the case $l=d$, which we derive in the appendix. The limits turn out to be a mixture of what one obtains in both settings separately. In the central limit theorem we get a conditionally Gaussian limit, where the conditional variance is a complicated functional of both the volatility $\sigma$ and the jumps of $X$. 
    
\begin{subsection}{Law of large numbers}
We will prove a law of large numbers for the quantity given in \eqref{mst}. As already mentioned we will need a combination of the methods from section \ref{tjc} and methods for U-statistics of continuous It\^o-semimartingales that were developed in \cite{PSZ}. We obtain the following result.
\begin{theorem}\label{lln2}
Let $H(\mathbf{x}, \mathbf{y})=|x_1|^{p_1} \cdots |x_l|^{p_l} |y_1|^{q_1} \ \cdots |y_{d-l}|^{q_{d-l}} L(\mathbf{x},\mathbf{y})$ with $p_1, \dots, p_l <2$ and $q_{1},\dots, q_{d-l}>2$ for some $0\leq l \leq d$. The function $L:\mathbb{R}^d \to \mathbb{R}$  is assumed to be continuous with $|L(\mathbf{x},\mathbf{y})| \leq u(\mathbf{y})$ for some $u \in \mathcal{C}(\mathbb{R}^{d-l})$. Then, for fixed $t>0$
\begin{align*}
Y_t^n(H,X,l) \toop Y_t(H,X,l)=\sum_{\mathbf{s} \in [0,t]^{d-l}} \int_{[0,t]^l} \rho_H(\sigma_{\mathbf{u}} ,\Delta X_{\mathbf{s}}) d\mathbf{u} ,
\end{align*}
where
\[
\rho_H(\mathbf{x} ,\mathbf{y})= \mathbb{E}[H(x_1 U_1, \dots, x_l U_l, \mathbf{y})]
\]
for arbitrary $\mathbf{x}\in \mathbb{R}^{l},\mathbf{y}\in \mathbb{R}^{d-l}$, and with $(U_1, \dots,U_l)\sim \mathcal{N}(\mathbf{0}, \operatorname{id}_l)$. 
\end{theorem}
\begin{remark}
In the special case $l=0$ we obtain the result from Theorem \ref{lln}. For $l=d$ we basically get the same limit as the case of U-statistics for continuous semimartingales $X$ (see Theorem 3.3 in \cite{PSZ}).  
\end{remark}
\begin{proof}
By the standard localization procedure we may assume that $X$ and $\sigma$ are bounded by a constant $A$. We will start by proving the following two assertions:
\begin{align}
\sup_{\mathbf{y} \in [-2A,2A]^{d-l}}\Big|\frac{1}{n^l} \sum_{\textbf{i}\in \mathcal{B}_t^n (l)} g(\sqrt{n} \Delta_{\mathbf{i}}^n X, \mathbf{y}) - \int_{[0,t]^l} \rho_g(\sigma_{\mathbf{u}} ,\mathbf{y}) d\mathbf{u} \Big| \toop 0, \label{glmy}\\
\sup_{\mathbf{x} \in [-A,A]^l}\Big| \sum_{\textbf{j}\in \mathcal{B}_t^n (d-l)} \rho_{H}(\mathbf{x}, \Delta_{\mathbf{j}}^n X) - \sum_{\mathbf{s} \in [0,t]^{d-l}} \rho_H(\mathbf{x} ,\Delta X_{\mathbf{s}} ) \Big|  \toop 0 \label{glmx},
\end{align}
where $g(\mathbf{x},\mathbf{y})=|x_1|^{p_1}\cdots |x_l|^{p_l}  L(\mathbf{x},\mathbf{y})$. The proofs mainly rely on the following decomposition for any real-valued function $f$ defined on some compact set $C \subset \mathbb{R}^k$: If $C' \subset C$ is finite and for any $\mathbf{x}\in C$ there exists $\mathbf{y}\in C'$ such that $\left\|\mathbf{x}-\mathbf{y}\right\|\leq \delta$ for some $\delta >0$, then
\[
\sup_{\mathbf{x}\in C} |f(\mathbf{x})| \leq \max_{\mathbf{x}\in C'} |f(\mathbf{x})|  + \sup_{\mathbf{x},\mathbf{y}\in C \atop\left\|\mathbf{x}-\mathbf{y}\right\|\leq \delta} |f(\mathbf{x}) - f(\mathbf{y})|.
\]
Now denote the continuous part of the semimartingale $X$ by $X^c$. For the proof of \eqref{glmy} we first observe that for fixed $\mathbf{y} \in \mathbb{R}^{d-l}$ we have
\[
\frac{1}{n^l} \sum_{\textbf{i}\in \mathcal{B}_t^n (l)} \big(g(\sqrt{n} \Delta_{\mathbf{i}}^n X, \mathbf{y})- g(\sqrt{n} \Delta_{\mathbf{i}}^n X^c, \mathbf{y})\big) \toop 0. 
\]
We will not give a detailed proof of this "elimination of jumps" step since it follows essentially from the corresponding known case $l=1$ (see  \cite[section 3.4.3]{JP}) in combination with the methods we use in the proof of \eqref{supdif}. Using the results of the asymptotic theory for U-statistics of continuous It\^o semimartingales given in \cite[Prop. 3.2]{PSZ} we further obtain (still for fixed $\mathbf{y}$)
\[
\frac{1}{n^l} \sum_{\textbf{i}\in \mathcal{B}_t^n (l)}  g(\sqrt{n} \Delta_{\mathbf{i}}^n X^c, \mathbf{y}) \toop \int_{[0,t]^l} \rho_g(\sigma_{\mathbf{u}} ,\mathbf{y}) d\mathbf{u}.
\]
To complete the proof of \eqref{glmy} we will show
\begin{align}\label{supdif}
\xi^n(m):=\sup_{\mathbf{x},\mathbf{y}\in [-2A,2A]^{d-l} \atop\left\|\mathbf{x}-\mathbf{y}\right\|\leq \frac{1}{m}} \frac{1}{n^l} \Big|\sum_{\textbf{i}\in \mathcal{B}_t^n (l)} \Big( g(\sqrt{n} \Delta_{\mathbf{i}}^n X, \mathbf{x}) - g(\sqrt{n} \Delta_{\mathbf{i}}^n X, \mathbf{y})\Big)\Big| \toop 0
\end{align}
if we first let $n$ and then $m$ go to infinity. The corresponding convergence of the integral term in \eqref{glmy} is easy and will therefore be omitted. 

Let $\epsilon >0$ be fixed such that $\max(p_1,\dots, p_l)+\epsilon <2$, and for all $\alpha >0$ and $k \in \mathbb{N}$ define the modulus of continuity
\[
\delta_k(\alpha) := \sup \left\{|g(\mathbf{u}, \mathbf{x})-g(\mathbf{u},\mathbf{y})| \Big| \left\|\mathbf{u}\right\| \leq k, \left\|(\mathbf{x},\mathbf{y})\right\|\leq 2A, \left\|\mathbf{x}-\mathbf{y}\right\| \leq \alpha \right\}.
\]  
Then we have
\begin{align*}
\xi^n (m) &\leq K \Bigg( \delta_k (m^{-1}) + \sup_{\mathbf{x},\mathbf{y}\in [-2A,2A]^{d-l} \atop\left\|\mathbf{x}-\mathbf{y}\right\|\leq \frac{1}{m}} \frac{1}{n^l} \sum_{\textbf{i}\in \mathcal{B}_t^n (l)} \mathbbm{1}_{\left\{\left\|\sqrt{n}\Delta_{\mathbf{i}}^n X \right\| \geq k \right\}} \Big(| g(\sqrt{n} \Delta_{\mathbf{i}}^n X, \mathbf{x})| + |g(\sqrt{n} \Delta_{\mathbf{i}}^n X, \mathbf{y})|\Big) \Bigg) \\
&\leq K \Big( \delta_k (m^{-1}) + \frac{1}{n^l} \sum_{\textbf{i}\in \mathcal{B}_t^n (l)} |\sqrt{n}\Delta_{i_1}^n X|^{p_1}\cdots  |\sqrt{n}\Delta_{i_l}^n X|^{p_l} \frac{|\sqrt{n}\Delta_{i_1}^nX|^{\epsilon}+\dots +|\sqrt{n}\Delta_{i_l}^nX|^{\epsilon} }{k^{\epsilon}} \Big) \\
& \toop K \Big( \delta_k (m^{-1}) + \frac{1}{k^{\epsilon}}\sum_{j=1}^l \prod_{i=1}^l \int_0^t m_{p_i+\delta_{ij}\epsilon} |\sigma_s|^{p_i + \delta_{ij} \epsilon} ds  \Big) \quad \text{as} \quad n \to \infty, 
\end{align*}
where $m_p$ is the $p$-th absolute moment of the standard normal distribution and $\delta_{ij}$ is the Kronecker delta (for a proof of the last convergence see \cite[Theorem 2.4]{J2}) . The latter expression obviously converges to $0$ if we let $m\to \infty$ and then $k \to \infty$, which completes the proof of $\eqref{glmy}$. 

We will prove \eqref{glmx} in a similar way. Since $\rho_{H}(\mathbf{x}, \mathbf{y})/ |y_1 \cdot \ldots \cdot y_{d-l}|^2\to 0$ as $\mathbf{y} \to 0$ , Theorem \ref{lln} implies 
\[
\sum_{\textbf{j}\in \mathcal{B}_t^n (d-l)} \rho_{H}(\mathbf{x}, \Delta_{\mathbf{j}}^n X) \toop \sum_{\mathbf{s} \in [0,t]^{d-l}} \rho_H(\mathbf{x} ,\Delta X_{\mathbf{s}} ),
\]
i.e.\ pointwise convergence for fixed $\mathbf{x} \in [-A,A]^l$. Moreover,
\begin{align*}
\sup_{\mathbf{x},\mathbf{y}\in [-A,A]^{l} \atop\left\|\mathbf{x}-\mathbf{y}\right\| \leq \frac{1}{m}} \sum_{\textbf{j}\in \mathcal{B}_t^n (d-l)}& \Big|\rho_{H}(\mathbf{x}, \Delta_{\mathbf{j}}^n X)-\rho_{H}(\mathbf{y}, \Delta_{\mathbf{j}}^n X)  \Big|  \\
&\leq \Bigg( \prod_{i=1}^{d-l} \sum_{j=1}^{\left\lfloor nt\right\rfloor} |\Delta_{j}^n X|^{q_i} \Bigg) \sup_{\mathbf{x},\mathbf{y}\in [-A,A]^{l} \atop\left\|\mathbf{x}-\mathbf{y}\right\| \leq \frac{1}{m}} \sup_{\left\|\mathbf{z}\right\| \leq 2A} \Big|\rho_{g}(\mathbf{x}, \mathbf{z}) -\rho_{g}(\mathbf{y}, \mathbf{z})  \Big|.  
\end{align*}
The term in brackets converges in probability to some finite limit by Theorem \ref{lln} as $n\to \infty$, and the supremum goes to $0$ as $m \to \infty$ because $\rho_g$ is continuous. By similar arguments it follows that
\begin{align*}
\sup_{\mathbf{x},\mathbf{y}\in [-A,A]^{l} \atop\left\|\mathbf{x}-\mathbf{y}\right\| \leq \frac{1}{m}} \sum_{\mathbf{s} \in [0,t]^{d-l}}& \Big|\rho_{H}(\mathbf{x}, \Delta X_{\mathbf{s}})-\rho_{H}(\mathbf{y}, \Delta X_{\mathbf{s}})  \Big| \toop 0,
\end{align*}
if we let $m$ go to infinity. Therefore \eqref{glmx} holds. 

We will now finish the proof of Theorem \ref{lln2} in two steps. First we have
\begin{align*}
&\Big|\frac{1}{n^l} \sum_{\textbf{i}\in \mathcal{B}_t^n (l)}\sum_{\textbf{j}\in \mathcal{B}_t^n (d-l)} H(\sqrt{n} \Delta_{\mathbf{i}}^n X, \Delta_{\mathbf{j}}^n X) - \sum_{\textbf{j}\in \mathcal{B}_t^n (d-l)}  \int_{[0,t]^l} \rho_H(\sigma_{\mathbf{u}} ,\Delta_{\mathbf{j}}^n X) d\mathbf{u} \Big| \\
\leq & \Bigg( \prod_{i=1}^{d-l} \sum_{j=1 }^{\left\lfloor nt\right\rfloor} |\Delta_{j}^n X|^{q_i} \Bigg)\sup_{\mathbf{y} \in [-2A,2A]^{d-l}}\Big|\frac{1}{n^l} \sum_{\textbf{i}\in \mathcal{B}_t^n (l)} g(\sqrt{n} \Delta_{\mathbf{i}}^n X, \mathbf{y}) - \int_{[0,t]^l} \rho_g(\sigma_{\mathbf{u}} ,\mathbf{y}) d\mathbf{u} \Big| \toop 0
\end{align*}
by \eqref{glmy}. From \eqref{glmx} we obtain the functional convergence
\[
\begin{pmatrix}
   (\sigma_s)_{0\leq s \leq t}\\
   \sum_{\textbf{j}\in \mathcal{B}_t^n (d-l)} \rho_{H}(\cdot, \Delta_{\mathbf{j}}^n X) 
\end{pmatrix} 
\toop
\begin{pmatrix}
   (\sigma_s)_{0\leq s \leq t}\\
   \sum_{\mathbf{s} \in [0,t]^{d-l}} \rho_H(\cdot ,\Delta X_{\mathbf{s}} ) 
\end{pmatrix} 
\]
in the space $\mathcal{D}([0,t]) \times \mathcal{C}([-A,A]^l)$. Define the mapping
\[
\Phi: \mathcal{D}([0,t]) \times \mathcal{C}(\mathbb{R}^l) \to \mathbb{R}, \quad (f,g) \longmapsto \int_{[0,t]^l} g(f(u_1), \dots, f(u_l)) d\mathbf{u}.
\]
This mapping is continuous and therefore we obtain by the continuous mapping theorem
\[
\sum_{\textbf{j}\in \mathcal{B}_t^n (d-l)}  \int_{[0,t]^l} \rho_H(\sigma_{\mathbf{u}} ,\Delta_{\mathbf{j}}^n X) d\mathbf{u} \toop \sum_{\mathbf{s} \in [0,t]^{d-l}} \int_{[0,t]^l} \rho_H(\sigma_{\mathbf{u}} ,\Delta X_{\mathbf{s}}) d\mathbf{u},
\]
which ends the proof. 
\end{proof}
\end{subsection}

\begin{subsection}{Central limit theorem}
In the mixed case we need some additional assumptions on the process $X$. First we assume that the volatility process $\sigma_t$ is not vanishing, i.e.\ $\sigma_t \neq 0$ for all $t\in [0,T]$, and that $\sigma$ is itself a continuous It\^o-semimartingale of the form
\[
\sigma_t=\sigma_0 + \int_0^{t} \tilde{b}_s ds + \int_0^t \tilde{\sigma}_s dW_s + \int_0^t \tilde{v}_s dV_s,
\]
where $\tilde{b}_s,\tilde{\sigma}_s$, and $\tilde{v}_s$ are c\`adl\`ag processes and $V_t$ is a Brownian motion independent of $W$. As a boundedness condition on the jumps we further require that there is a sequence $\Gamma_k: \mathbb{R} \to \mathbb{R}$ of functions and a localizing sequence $(\tau_k)_{k \in \mathbb{N}}$ of stopping times such that $|\delta(\omega,t,z)| \wedge 1\leq \Gamma_k(z)$ for all $\omega,t$ with $t \leq \tau_k(\omega)$ and 
\[
\int \Gamma_k(z)^{r}\lambda(dz) < \infty
\]  
for some $0<r<1$. In particular, the jumps of the process $X$ are then absolutely summable.

The central limit theorem will again be stable under stopping, so we can assume without loss of generality that there is a function $\Gamma :\mathbb{R} \to \mathbb{R}$ and a constant $A$ such that $\delta(\omega, t,z) \leq \Gamma(z)$ and
\[
\sup\{|X_t(\omega)|, |b_t(\omega)|, |\sigma_t(\omega)|, |\sigma_t^{-1}(\omega)|, |\tilde{b}_t(\omega)|, |\tilde{\sigma}_t(\omega)|, |\tilde{v}_t(\omega)|\}\leq A,
\]
uniformly in $(\omega,t)$. We may further assume $\Gamma(z) \leq A$ for all $z\in \mathbb{R}$ and
\[
\int \Gamma(z)^r \lambda(dz) < \infty.
\]
Before we state the central limit theorem for $\sqrt{n}(Y_t^n(H,X,l)- Y_t(H,X,l))$ we give a few auxiliary results. A typical procedure in proofs of results such as Theorem \ref{cltmc} is to replace the scaled increments of $X$ (for us: the terms in the first $l$ arguments) by the first order approximation $\alpha_i^n := \sqrt{n} \sigma_{\frac{i-1}{n}} \Delta_i^n W$ of the continuous part of $X$. In combination with other simplifications, this procedure will lead to asymptotic equivalence of $\sqrt{n}(Y_t^n(H,X,l)- Y_t(H,X,l))$  with
\[
\sum_{\mathbf{q}: S_{\mathbf{q}}\leq t}\Bigg(\frac{1}{n^l} \sum_{\mathbf{i} \in \mathcal{B}_t^n (l)}\sum_{k=l+1}^d \partial_k H\big(\alpha_{\mathbf{i}}^n, \Delta X_{S_{\mathbf{q}}}\big)R(n,q_k)+ \sqrt{n}\Big( \frac{1}{n^l} \sum_{\mathbf{i} \in \mathcal{B}_t^n (l) } H\big(\alpha_{\mathbf{i}}^n,\Delta X_{S_{\mathbf{q}}} \big) - \int_{[0,t]^l} \rho_H (\sigma_{\mathbf{s}},  \Delta X_{S_{\mathbf{q}}}) d\mathbf{s}\Big) \Bigg).
\]
For now, consider only the term in brackets, with $R(n,q_k) \equiv 1$ for simplicity. We can see that if $\Delta X_{S_{\mathbf{q}}}$ was just a deterministic number, we could derive the limit by using the asymptotic theory for U-statistics developed in \cite{PSZ}. For the first summand we would need a law of large numbers and for the second one a central limit theorem. Since $\Delta X_{S_{\mathbf{q}}}$ is of course in general not deterministic, the above decomposition indicates that it might be useful to have the theorems for U-statistics uniformly in some additional variables. 
As a first result in that direction we have the following claim.
\begin{proposition}\label{ulln}
Let $0 \leq l \leq d$ and $G: \mathbb{R}^l \times [-A,A]^{d-l} \to \mathbb{R}$ be a continuous function that is of polynomial growth in the first $l$ arguments, i.e.\ $|G(\mathbf{x},\mathbf{y})|\leq (1+\left\|\mathbf{x}\right\|^p)w(\mathbf{y})$ for some $p \geq 0$ and $w \in \mathcal{C}([-A,A]^{d-l})$. Then
\[
\mathbb{B}_t^n(G,\mathbf{x}):=\frac{1}{n^l} \sum_{\mathbf{i} \in \mathcal{B}_t^n (l) } G\big(\alpha_{\mathbf{i}}^n, \mathbf{y}\big) \toop \mathbb{B}_t(G,\mathbf{y}):= \int_{[0,t]^l} \rho_G (\sigma_{\mathbf{s}}, \mathbf{y}) d\mathbf{s}
\]
in the space $\mathcal{C}([-A,A]^{d-l})$, where 
\[
\rho_G (\mathbf{x},\mathbf{y}) := \mathbb{E}[G(x_1U_1,\ldots,x_l U_l, \mathbf{y})]
\]
for a standard normal variable $U=(U_1,\ldots, U_l)$. 
\end{proposition}

\begin{proof} This result follows exactly in the same way as \eqref{glmy} without the elimination of jumps step in the beginning. 
\end{proof}

In addition to this functional law of large numbers we further need the associated functional central limit theorem for
\begin{align}\label{clte}
\mathbb{U}_t^n(G,\mathbf{y})=\sqrt{n}\Big( \frac{1}{n^l} \sum_{\mathbf{i} \in \mathcal{B}_t^n (l) } G\big(\alpha_{\mathbf{i}}^n, \mathbf{y}\big) - \int_{[0,t]^l} \rho_G (\sigma_{\mathbf{s}}, \mathbf{y}) d\mathbf{s}\Big),
\end{align}
In order to obtain a limit theorem we will need to show tightness and the convergence of the finite dimensional distributions. We will use that, for fixed $\mathbf{y}$, an asymptotic theory for \eqref{clte} is given in \cite[Prop. 4.3]{PSZ}, but under too strong assumptions on the function $G$ for our purpose. In particular, we weaken the assumption of differentiability of $G$ in the following proposition whose proof can be found in the appendix.

\begin{proposition}\label{ustab}  Let $0 \leq l \leq d$ and let $G: \mathbb{R}^d \to \mathbb{R}$ be a function that is even in the first $l$ arguments and can be written in the form $G(\mathbf{x},\mathbf{y})=|x_1|^{p_1}\cdots |x_l|^{p_l} L(\mathbf{x},\mathbf{y})$ for some function $L \in \mathcal{C}^{d+1}(\mathbb{R}^{d}) $ and constants $p_1,\dots,p_{l} \in \mathbb{R}$ with $0<p_1, \dots, p_l<1$. We further impose the following growth conditions: 
\begin{align}\label{growth}
&|L(\mathbf{x},\mathbf{y})|\leq u(\mathbf{y}),  \quad  \Big| \partial_{ii}^2 L(\mathbf{x},\mathbf{y})\Big| \leq (1+\left\|\mathbf{x}\right\|^{\beta_i})u(\mathbf{y}) \quad (1\leq i\leq d),\\
&\Big| \partial_{j_1} \cdots \partial_{j_k}L(\mathbf{x},\mathbf{y})\Big| \leq (1+\left\|\mathbf{x}\right\|^{\gamma_{j_1\dots j_k}})u(\mathbf{y}), \quad  (1\leq k \leq d;\ 1\leq j_1 < \dots < j_k \leq d)
\end{align}
for some constants $\beta_i, \gamma_{j_1\dots j_k} \geq 0$, and a function $u \in \mathcal{C}(\mathbb{R}^{d-l})$. The constants are assumed to fulfill $\gamma_j +p_i<1$ for $i\neq j$ and $i=1,\dots,l$, $j=1,\dots,d$. Then we have, for a fixed $t>0$
\begin{align}
(\mathbb{U}_t^n(G,\cdot), (R_{-}(n,p),R_{+}(n,p))_{p \geq 1})) \stab (\mathbb{U}_t(G,\cdot), (R_{p-},R_{p+})_{p\geq 1}) 
\end{align}
in the space $\mathcal{C}([-A,A]^{d-l}) \times \mathbb{R}^{\mathbb{N}}\times \mathbb{R}^{\mathbb{N}}$, where $(\mathbb{U}_t(G,\cdot),(R_{p-},R_{p+})_{p\geq 1})$ is defined on an extension $(\tilde{\Omega}, \tilde{\mathcal{F}}, \tilde{\mathcal{P}})$ of the original probability space, $\mathbb{U}_t(G,\cdot)$ is $\mathcal{F}$-conditionally independent of $(\kappa_k, \psi_{k\pm})_{k\geq1}$ and $\mathcal{F}$-conditionally centered Gaussian with covariance structure 
\begin{align}\label{cov}
C(\mathbf{y},\mathbf{y}'):=&\mathbb{E}[\mathbb{U}_t(G,\mathbf{y})\mathbb{U}_t(G,\mathbf{y}')| \mathcal{F}]\\
=&\sum_{i,j=1}^l\int_0^t \Big(\int_{\mathbb{R}} f_i(u,\mathbf{y})f_j(u,\mathbf{y'}) \phi_{\sigma_s}(u) du -\Big(\int_{\mathbb{R}}f_i(u,\mathbf{y}) \phi_{\sigma_s}(u) du \Big) \Big(\int_{\mathbb{R}} f_j(u,\mathbf{y}') \phi_{\sigma_s}(u) du \Big) ds\Big) \nonumber
\end{align} 
where 
\[
f_i(u,\mathbf{y})=\int_{[0,t]^{l-1} }\int_{\mathbb{R}^{l-1}} G(\sigma_{s_1}v_1,\dots,\sigma_{s_{i-1}}v_{i-1},u,\sigma_{s_{i+1}}v_{i+1},\dots, \sigma_{s_{l}}v_{l}, \mathbf{y}) \phi(\mathbf{v})d\mathbf{v} d\mathbf{s}.
\]
\end{proposition}
\begin{remark}
The proposition is stated for the approximations $\alpha_i^n$ of the increments of $X$. We remark that the result is still true in the finite dimensional distribution sense if we replace $\alpha_i^n$ by the increments $\Delta_i^n X$. This follows by the same arguments as the elimination of jumps step in Theorem \ref{cltmc} and Proposition \ref{appal}.    
\end{remark}

We will now state the main theorem of this section. After some approximation steps the proof will mainly consist of an application of the previously established methods in combination with the continuous mapping theorem. 

\begin{theorem}\label{cltmc}
Let $0 \leq l \leq d$ and $H: \mathbb{R}^d \to \mathbb{R}$ be a function that is even in the first $l$ arguments and can be written in the form $H(\mathbf{x},\mathbf{y})=|x_1|^{p_1}\cdots |x_l|^{p_l} |y_1|^{q_1} \cdots |y_{d-l}|^{q_{d-l}}L(\mathbf{x},\mathbf{y})$ for some function $L \in \mathcal{C}^{d+1}(\mathbb{R}^{d}) $ and constants $p_1,\dots,p_{l}, q_1,\dots, q_{d-l} \in \mathbb{R}$ with $0<p_1, \dots, p_l<1$ and $q_1,\dots, q_{d-l} >3$. We further assume that $L$ fulfills the same assumptions as in Proposition \ref{ustab}. Then we have, for a fixed $t>0$
\begin{align*}
\sqrt{n}\Big(Y_t^n&(H,X,l)-Y_t(H,X,l)\Big) \\
&\stab V'(H,X,l)_t=\sum_{\mathbf{k}:T_\mathbf{k}\leq t}\Big(  \sum_{j=l+1}^d\int_{[0,t]^l}\rho_{\partial_j H} (\sigma_{\mathbf{u}}, \Delta X_{T_\mathbf{k}})  d\mathbf{u} R_{k_j}+ \mathbb{U}_t (H,\Delta X_{T_\mathbf{k}}) \Big).
\end{align*}
The limiting process is $\mathcal{F}$-conditionally centered Gaussian with variance
\begin{align}\label{covm}
\mathbb{E}[(V'(H,X,l)_t)^2|\mathcal{F}]=\sum_{s \leq t} \Big(\sum_{k=l+1}^d \tilde{V}_k(H,X,l, \Delta X_s)\Big)^2 \sigma_s^2 + \sum_{\mathbf{s_1},\mathbf{s_2}\in [0,t]^{d-l}} C(\Delta X_{\mathbf{s_1}},\Delta X_{\mathbf{s_2}}),
\end{align}
where the function $C$ is given in \eqref{cov} and
\[
\tilde{V}_k(H,X,l, y)=\sum_{s_{l+1},\dots, s_{k-1},s_{k+1},\dots, s_d\leq t} \int_{[0,t]^l} \rho_{\partial_k H} (\sigma_{\mathbf{u}}, \Delta X_{s_{l+1}},\dots,\Delta X_{s_{k-1}},y,\Delta X_{s_{k+1}},\dots,\Delta X_{s_{d}}) d\mathbf{u}.
\]
Furthermore the $\mathcal{F}$-conditional law of the limit does not depend on the choice of the sequence $(T_k)_{k \in \mathbb{N}}$.
\end{theorem}

\begin{remark}
The result coincides with the central limit theorem in section \ref{tjc} if $l=0$, but under stronger assumptions. In particular the assumed continuity of $\sigma$ yields that the limit is always conditionally Gaussian. We further remark that the theorem also holds in the finite distribution sense in $t$. 
\end{remark}

\begin{proof}
In the first part of the proof we will eliminate the jumps in the first argument. We split $X$ into its continuous part $X^c$ and the jump part $X^d=\delta \ast \mathfrak{p}$ via $X=X_0+X^c+X^d$. Note that $X^d$ exists since the jumps are absolutely summable under our assumptions. We will now show that
\[
\xi_n = \sqrt{n}\Big(\frac{1}{n^l}\sum_{\textbf{i}\in \mathcal{B}_t^n (l)}\sum_{\textbf{j}\in \mathcal{B}_t^n (d-l)} H(\sqrt{n} \Delta_{\mathbf{i}}^n X, \Delta_{\mathbf{j}}^n X) - \frac{1}{n^l}\sum_{\textbf{i}\in \mathcal{B}_t^n (l)}\sum_{\textbf{j}\in \mathcal{B}_t^n (d-l)} H(\sqrt{n} \Delta_{\mathbf{i}}^n X^c, \Delta_{\mathbf{j}}^n X)\Big) \toop 0. 
\]
Observe that under our growth assumptions on $L$ we can deduce  
\begin{align}\label{ineq}
|L(\mathbf{x}+\mathbf{z},\mathbf{y})-L(\mathbf{x},\mathbf{y})| \leq  K u(\mathbf{y}) (1+\sum_{i=1}^l \left\|\mathbf{x}\right\|^{\gamma_i}) \sum_{j=1}^l |z_j|^{p_j}
\end{align}
This inequality trivially holds if $\left\|\mathbf{z}\right\|>1$ because $\left\|L(\mathbf{x},\mathbf{y})\right\|\leq u(\mathbf{y})$. In the case $\left\|\mathbf{z}\right\|\leq 1$ we can use the mean value theorem in combination with $|z|\leq |z|^p$ for $|z|\leq 1$ and $0<p<1$. Since we also have $\big||x_i+z_i|^{p_i}-|x_i|^{p_i}\big| \leq |z_i|^{p_i}$ for $1\leq i \leq l$, we have, with $\mathbf{q}=(q_1,\dots,q_{d-l})$, the estimate 
\begin{align*}
|H(\mathbf{x}+\mathbf{z},\mathbf{y})-H(\mathbf{x},\mathbf{y})| \leq Ku(\mathbf{y}) |\mathbf{y}|^{\mathbf{q}} \sum_{\mathbf{m}}P_{\mathbf{m}} (\mathbf{x}) |\mathbf{z}|^{\mathbf{m}} 
\end{align*}
where $P_{\mathbf{m}} \in \mathfrak{P}(l)$ (see \eqref{power} for a definition) and the sum runs over all $\mathbf{m}=(m_1,\dots, m_l)\neq (0,\dots,0)$ with $m_j$ either $p_j$ or $0$. We do not give an explicit formula here since the only important property is $\mathbb{E}[P_{\mathbf{m}}(\sqrt{n} \Delta_{\mathbf{i}}^n X')^q] \leq K$ for all $q \geq 0$, which directly follows from the Burkholder inequality. 
 Because of the boundedness of $X$ and the continuity of $u$ this leads to the following bound on $\xi_n$:
\[
|\xi_n| \leq \Big(K\sum_{\textbf{j}\in \mathcal{B}_t^n (d-l)}|\Delta_{\mathbf{j}}^n X|^{\mathbf{q}}\Big) \Big(\frac{\sqrt{n}}{n^l}\sum_{\textbf{i}\in \mathcal{B}_t^n (l)} \sum_{\mathbf{m}}P_{\mathbf{m}} (\sqrt{n} \Delta_{\mathbf{i}}^n X^c) |\sqrt{n} \Delta_{\mathbf{i}}^n X^d|^{\mathbf{m}}\Big).
\]
The first factor converges in probability to some finite limit, and hence it is enough to show that the second factor converges in $L^1$ to $0$. Without loss of generality we restrict ourselves to the summand with $\mathbf{m}=(p_1,\dots, p_k, 0,\dots,0)$ for some $1\leq k \leq l$. From \cite[Lemma 2.1.7]{JP} it follows that 
\begin{align}\label{jie}
 \mathbb{E}[ | \Delta_i^n X^d|^{q}| \mathcal{F}_{\frac{i-1}{n}}] \leq \frac{K}{n} \quad \text{for all } q >0. 
\end{align}

Let $r:=\max_{1\leq i \leq l} p_i$ and $b_k(\mathbf{i}):=\#\left\{i_1,\dots,i_k \right\}$ for $\mathbf{i}=(i_1,\dots, i_l)$. Note that the number of $\mathbf{i} \in \mathcal{B}_t^n (l)$ with $b_k(\mathbf{i})=m$ is of order $n^{m+l-k}$ for $1\leq m \leq k$. An application of H\"older inequality, successive use of \eqref{jie} and the boundedness of $X$ gives 
\begin{align*}
&\mathbb{E}\Big(\frac{\sqrt{n}}{n^l}\sum_{\textbf{i}\in \mathcal{B}_t^n (l)} P_{\mathbf{m}} (\sqrt{n} \Delta_{\mathbf{i}}^n X^c) |\sqrt{n} \Delta_{i_1}^n X^d|^{p_1}\dots |\sqrt{n} \Delta_{i_k}^n X^d|^{p_k}\Big) \\
\leq &\frac{n^{1/2+kr/2}}{n^l} \sum_{\textbf{i}\in \mathcal{B}_t^n (l)} \Bigg(\mathbb{E}[P_{\mathbf{m}} (\sqrt{n} \Delta_{\mathbf{i}}^n X^c)^{\frac{4}{1-r}} ] \Bigg)^{\frac{1-r}{4}}\Big(\mathbb{E}\Big[ \Big( |\Delta_{i_1}^n X^d|^{p_1}\dots |\Delta_{i_k}^n X^d|^{p_k}\Big)^{\frac{4r}{3+r}} \Big]\Big)^{\frac{3+r}{4}}\\
\leq & K\frac{n^{1/2+kr/2}}{n^l}\sum_{\textbf{i}\in \mathcal{B}_t^n (l)} n^{-b_k (\mathbf{i}) (3+r)/4}\leq K\frac{n^{1/2+kr/2}}{n^l}\sum_{j=1}^k n^{-j (3+r)/4} n^{j+l-k}=K\sum_{j=1}^k n^{(2-2k+(2k-j)(r-1))/4}.  
\end{align*}
The latter expression converges to $0$ since $r<1$.

In the next step we will show that we can replace the increments $\Delta_i^n X^c$ of the continuous part of $X$ by their first order approximation $\alpha_i^n=\sqrt{n}\sigma_{\frac{i-1}{n}}\Delta_i^n W$.
\begin{proposition}\label{appal} It holds that
\begin{align*}
\xi_n' = \sqrt{n}\Big(\frac{1}{n^l}\sum_{\textbf{i}\in \mathcal{B}_t^n (l)}\sum_{\textbf{j}\in \mathcal{B}_t^n (d-l)} H(\sqrt{n}\Delta_{\mathbf{i}}^n X^c, \Delta_{\mathbf{j}}^n X)-\frac{1}{n^l}\sum_{\textbf{i}\in \mathcal{B}_t^n (l)}\sum_{\textbf{j}\in \mathcal{B}_t^n (d-l)}  H(\alpha_{\mathbf{i}}^n, \Delta_{\mathbf{j}}^n X) \Big) \toop 0
\end{align*}
as $n \to \infty$.
\end{proposition}
We shift the proof of this result to the appendix. Having simplified the statistics in the first argument, we now focus on the second one, more precisely on the process
\[
\theta_n(H) = \sqrt{n}\Big(\frac{1}{n^l}\sum_{\textbf{i}\in \mathcal{B}_t^n (l)}\sum_{\textbf{j}\in \mathcal{B}_t^n (d-l)}  H(\alpha_{\mathbf{i}}^n, \Delta_{\mathbf{j}}^n X) -\frac{1}{n^l}\sum_{\textbf{i}\in \mathcal{B}_t^n (l)}\sum_{\mathbf{s} \in [0,t]^{d-l}}  H(\alpha_{\mathbf{i}}^n, \Delta X_{\mathbf{s}}) \Big).
\] 
In the following we will use the notation from section \ref{cltjc}. We split $\theta_n(H)$ into 
\[
\theta_n(H) = \mathbbm{1}_{\Omega_n(m)} \theta_n(H) + \mathbbm{1}_{\Omega \backslash \Omega_n(m)} \theta_n(H).
\] 
Since $\Omega_n(m) \toop \Omega$ as $n \to \infty$, the latter term converges in probability to $0$ as $n\to \infty$. The following result will be shown in the appendix as well.
\begin{proposition}\label{appsp}
We have the convergence
\begin{align*} 
\mathbbm{1}_{\Omega_n(m)} \theta_n(H) -\frac{1}{n^l}\sum_{\textbf{i}\in \mathcal{B}_t^n (l)}\sum_{\mathbf{q} \in \mathcal{P}_t^n(m)^{d-l}} \sum_{k=l+1}^d \partial_k H(\alpha_\mathbf{i}^n , \Delta X_{S_\mathbf{q}})R(n,q_k) \toop 0
\end{align*}
if we first let $n\to \infty$ and then $m \to \infty$.
\end{proposition}
Using all the approximations, in view of Lemma \ref{lem6} we are left with
\begin{align*}
\Phi_t^n(m):=& \frac{1}{n^l}\sum_{\textbf{i}\in \mathcal{B}_t^n (l)}\sum_{\mathbf{q} \in \mathcal{P}_t^n(m)^{d-l}} \sum_{k=l+1}^d \partial_k H(\alpha_\mathbf{i}^n , \Delta X_{S_\mathbf{q}})R(n,q_k) +  \sum_{\mathbf{s} \in [0,t]^{d-l}} \mathbb{U}_t^n (H, \Delta X_\mathbf{s})  \\
=& \sum_{\mathbf{q} \in \mathbb{N}^{d-l}} \Big( \mathbbm{1}_{ \mathcal{P}_t^n(m)^{d-l}}(\mathbf{q}) \sum_{k=l+1}^d  \mathbb{B}_t^n (\partial_k H, \Delta X_{S_{\mathbf{q}}})R(n,q_k) + \mathbb{U}_t^n (H, \Delta X_{S_\mathbf{q}})\Big).
\end{align*}
The remainder of the proof will consist of four steps. First we use for all $k\in \mathbb{N}$ the decomposition  $\Phi_t^n (m)=\Phi_t^n(m,k)+\tilde{\Phi}_t^n (m,k)$, where
\[
\Phi_t^n(m,k):= \sum_{q_1,\dots, q_{d-l}\leq k}  \mathbbm{1}_{ \mathcal{P}_t^n(m)^{d-l}}(\mathbf{q}) \sum_{k=l+1}^d  \mathbb{B}_t^n (\partial_k H, \Delta X_{S_{\mathbf{q}}})R(n,q_k) + \sum_{\mathbf{q} \in \mathbb{N}^{d-l}}\mathbb{U}_t^n (H, \Delta X_{S_\mathbf{q}}), 
\]
i.e.\ we consider only finitely many jumps in the first summand. We will successively show 
\begin{align}
\lim_{k \to \infty} \limsup_{n \to \infty} \mathbb{P} (| \tilde{\Phi}_t^n (m,k)| >\eta) =0 \quad \text{for all} \quad \eta >0,\label{pr1} \\
\Phi_t^n (m,k) \stab \Phi_t (m,k) \quad \text{as} \quad n \to \infty, \label{pr2}
\end{align}
for a process $\Phi_t (m,k)$ that will be defined in \eqref{phimk}. Finally, with $\Phi_t(m)$ defined in \eqref{phim} we will show
\begin{align}
\Phi_t(m,k) \toop \Phi_t(m) \quad \text {as} \quad k \to \infty, \label{pr3} \\
\Phi_t (m) \toop V'(H,X,l)_t. \label{pr4}
\end{align}
For \eqref{pr1} observe that we have $\mathcal{P}_t^n(m) \subset \mathcal{P}_t(m)$ and therefore
\begin{align*}
 \mathbb{P}\Big( | \tilde{\Phi}_t^n (m,k)| > \eta\Big) \leq  \mathbb{P}\big(  \big\{ \omega: \mathcal{P}_t (m,\omega) \not\subset \left\{ 1,\dots,k \right\} \big\} \big) \to 0 \quad \text{as} \quad k \to \infty,
\end{align*}
since the sets $\mathcal{P}_t (m,\omega)$ are finite for fixed $\omega$ and $m$. For \eqref{pr2} recall that $g$ was defined by $g(\mathbf{x},\mathbf{y})=|x_1|^{p_1}\cdots |x_l|^{p_l} L(\mathbf{x},\mathbf{y})$. By Propositions \ref{ulln} and \ref{ustab} and from the properties of stable convergence (in particular, we need joint stable convergence with sequences converging in probability, which is useful for the indicators below) we have
\begin{align*}
(\mathbb{U}_t^n (g,\cdot), (\mathbb{B}_t^n (\partial_j H,\cdot))_{j=l+1}^d, &(\Delta X_{S_p})_{p \in \mathbb{N}},(R(n,p))_{p \in \mathbb{N}}, (\mathbbm{1}_{\mathcal{P}_t^n(m)}(p))_{p \in \mathbb{N}}) \\
&\stab (\mathbb{U}_t (g,\cdot), (\mathbb{B}_t (\partial_j H,\cdot))_{j=l+1}^d, (\Delta X_{S_p})_{p \in \mathbb{N}},(R_p)_{p \in \mathbb{N}}, (\mathbbm{1}_{\mathcal{P}_t(m)}(p))_{p \in \mathbb{N}})
\end{align*}
as $n\to \infty$ in the space $\mathcal{C}[-A,A]^{(d-l)}\times (\mathcal{C}[-A,A]^{(d-l)})^{d-l} \times \ell^{2}_A \times \mathbb{R}^{\mathbb{N}}\times \mathbb{R}^{\mathbb{N}}$, where we denote by $\ell^2_A$ the metric space 
\[
\ell^2_A:= \left\{(x_k)_{k \in \mathbb{N}} \in \ell^2 \ ; \  |x_k| \leq A \ \text{for all} \ k \in \mathbb{N} \right\}.
\]
For $k \in \mathbb{N}$ we now define a continuous mapping on $\mathcal{C}[-A,A]^{(d-l)}\times (\mathcal{C}[-A,A]^{(d-l)})^{d-l} \times \ell^{2}_A \times \mathbb{R}^{\mathbb{N}}\times \mathbb{R}^{\mathbb{N}}$ into the real numbers via
\begin{align*}
\phi_k (f,(g_r)_{r=1}^{d-l}, (x_j)_{j \in \mathbb{N}},(y_j)_{j \in \mathbb{N}},(z_j)_{j \in \mathbb{N}})= \sum_{j_1,\dots,j_{d-l}=1}^k & z_{j_1}\cdots z_{j_{d-l}} \sum_{r=l+1}^d g_r(x_{j_1},\dots,x_{j_{d-l}}) y_{j_r} \\
& + \sum_{j_1,\dots,j_{d-l}=1}^{\infty} |x_{j_1}|^{q_1}\cdots |x_{j_{d-l}}|^{q_{d-l}} f(x_{j_1},\dots, x_{j_{d-l}}).  
\end{align*}
The continuous mapping theorem then yields
\begin{align} \label{phimk}
\Phi_t^n(m,k)&=\phi_k(\mathbb{U}_t^n (g,\cdot), (\mathbb{B}_t^n (\partial_r H,\cdot))_{r=l+1}^d, (\Delta X_{S_p})_{p \in \mathbb{N}},(R(n,p))_{p \in \mathbb{N}}, (\mathbbm{1}_{\mathcal{P}_t^n(m)}(p))_{p \in \mathbb{N}}) \nonumber \\
&\stab  \phi_k(\mathbb{U}_t (g,\cdot), (\mathbb{B}_t (\partial_r H,\cdot))_{r=l+1}^d, (\Delta X_{S_p})_{p \in \mathbb{N}},(R_p)_{p \in \mathbb{N}}, (\mathbbm{1}_{\mathcal{P}_t(m)}(p))_{p \in \mathbb{N}}) \nonumber \\
& \quad = \sum_{q_1,\dots, q_{d-l}\leq k}  \mathbbm{1}_{ \mathcal{P}_t(m)^{d-l}}(\mathbf{q}) \sum_{r=l+1}^d  \mathbb{B}_t (\partial_r H, \Delta X_{S_{\mathbf{q}}})R(n,q_r) + \sum_{\mathbf{q} \in \mathbb{N}^{d-l}}\mathbb{U}_t (H, \Delta X_{S_\mathbf{q}})=:\Phi_t(m,k).
\end{align}
For $k \to \infty$ we have
\begin{align}\label{phim}
\Phi_t(m,k) \toas  \Phi_t(m):=\sum_{\mathbf{q} \in \mathbb{N}^{d-l}}\Big(  \mathbbm{1}_{ \mathcal{P}_t(m)^{d-l}}(\mathbf{q}) \sum_{r=l+1}^d  \mathbb{B}_t (\partial_r H, \Delta X_{S_{\mathbf{q}}})R(n,q_r) + \sum_{\mathbf{q} \in \mathbb{N}^{d-l}}\mathbb{U}_t (H, \Delta X_{S_\mathbf{q}})\Big),
\end{align}
i.e.\ \eqref{pr3}. For the last assertion \eqref{pr4} we have 
\begin{align*}
\mathbb{P}(|\Phi_t(m)-V_t'(H,X,l)|>\eta) & \leq K \mathbb{E}[(\Phi_t(m)-V_t'(H,X,l))^2]= K \mathbb{E}[\mathbb{E}[(\Phi_t(m)-V_t'(H,X,l))^2|\mathcal{F}]] \\
& \leq K \mathbb{E}[\sum_{\mathbf{k} \in \mathbb{N}^{d-l}} \sum_{r=l+1}^d (1-\mathbbm{1}_{ \mathcal{P}_t(m)^{d-l}}(\mathbf{k})) |\mathbb{B}_t(\partial_r H, \Delta X_{S_\mathbf{k}})|^2] \\
& \leq K \mathbb{E} [\sum_{\mathbf{k} \in \mathbb{N}^{d-l}} \sum_{r=l+1}^d (1-\mathbbm{1}_{ \mathcal{P}_t(m)^{d-l}}(\mathbf{k})) \prod_{i=1}^{d-l}\big(|\Delta X_{S_{k_i}}|^{q_i} +|\Delta X_{S_{k_i}}|^{q_i-1}\big)^2  ]. 
\end{align*}
Since the jumps are absolutely summable and bounded the latter expression converges to $0$ as $m \to \infty$.
\end{proof}
\end{subsection}
\end{section}

\begin{section}{Appendix}\label{appen}
\begin{subsection}{Existence of the limiting processes}
We give a proof that the limiting processes in Theorem \ref{stabvj} and Theorem \ref{cltmc} are well-defined. The proof will be similar to the proof of \cite[Prop. 4.1.4]{JP}. We restrict ourselves to proving that
\begin{align}\label{wd}
\sum_{\mathbf{k} : T_{\mathbf{k}}\leq t} \int_{[0,t]^l} \rho_{\partial_{l+1} H}(\sigma_{\mathbf{u}},\Delta X_{T_{\mathbf{k}}}) d\mathbf{u} R_{k_1}
\end{align}
is defined in a proper way, corresponding to Theorem \ref{cltmc}. For $l=0$ we basically get the result for Theorem \ref{stabvj}, but under slightly stronger assumptions. The proof, however, remains the same. 

We show that the sum in \eqref{wd} converges in probability for all $t$ and that the conditional properties mentioned in the theorems are fulfilled. Let $I_m(t)=\left\{ n: 1\leq n \leq m, T_n \leq t\right\}$. Define
\[
Z(m)_t := \sum_{\mathbf{k} \in I_m(t)^{d-l}} \int_{[0,t]^l} \rho_{\partial_{l+1} H}(\sigma_{\mathbf{u}},\Delta X_{T_{\mathbf{k}}}) d\mathbf{u} R_{k_1}.
\]
By fixing $\omega \in \Omega$, we further define the process $Z^{\omega}(m)_t$ on $(\Omega',\mathcal{F}',\mathbb{P}')$ by $Z^{\omega}(m)_t (\omega')=Z(m)_t(\omega,\omega')$. The process is obviously centered, and we can immediately deduce 
\begin{align}\label{var}
\mathbb{E}'(Z^{\omega}(m)_t^2) =  \sum_{k_1 \in I_m(t)} \Big(\sum_{\mathbf{k} \in I_m(t)^{d-l-1}} \int_{[0,t]^l} \rho_{\partial_{l+1} H}(\sigma_{\mathbf{u}},\Delta X_{T_{k_1}},\Delta X_{T_{\mathbf{k}}}) d\mathbf{u} \Big)^2 \sigma_{T_{k_1}}^2,
\end{align}
\begin{align} \label{char}
\mathbb{E}'(e^{i u Z^{\omega}(m)_t}) = \prod_{k_1 \in I_m(t)} \int e^{i u \sum_{\mathbf{k} \in I_m(t)^{d-l-1}} \int_{[0,t]^l} \rho_{\partial_{l+1} H}(\sigma_{\mathbf{u}},\Delta X_{T_{k_1}},\Delta X_{T_{\mathbf{k}}}) d\mathbf{u} R_{k_1}} d\mathbb{P}'.
\end{align}
The processes $X$ and $\sigma$ are both c\`adl\`ag and hence bounded on $[0,T]$ for a fixed $\omega \in \Omega$. Let now $m,m' \in \mathbb{N}$ with $m' \leq m$ and observe that $I_m(t)^{q}\backslash I_{m'}(t)^q \subset I_m(T)^q\backslash I_{m'}(T)^q$ for all $q \in \mathbb{N}$ and $t\leq T$. Since $L$ and $\partial_1 L$ are bounded on compact sets, we obtain 
\begin{align*}
&\mathbb{E}'\Big[\Big(\sup_{t \in [0,T]} |Z^{\omega}(m)_t-Z^{\omega}(m')_t|\Big)^2\Big] \\
= & \mathbb{E}'\Big[\Big(\sup_{t \in [0,T]} \Big|\sum_{\mathbf{k} \in I_m(t)^{d-l}\backslash I_{m'}(t)^{d-l}} \int_{[0,t]^l} \rho_{\partial_{l+1} H}(\sigma_{\mathbf{u}},\Delta X_{T_{\mathbf{k}}}) d\mathbf{u} R_{k_1}\Big|\Big)^2\Big]\\
\leq & \mathbb{E}'\Big[\Big( \sum_{\mathbf{k} \in I_m(T)^{d-l}\backslash I_{m'}(T)^{d-l}} \int_{[0,T]^l} \big|\rho_{\partial_{l+1} H}(\sigma_{\mathbf{u}},\Delta X_{T_{\mathbf{k}}})\big| d\mathbf{u} |R_{k_1}|\Big)^2\Big]\\
\leq & K(\omega) \Big( \sum_{\mathbf{k} \in I_m(T)^{d-l}\backslash I_{m'}(T)^{d-l}} \big(|\Delta X_{T_{k_1}}|^{q_1-1}+|\Delta X_{T_{k_1}}|^{q_1}\big) |\Delta X_{T_{k_2}}|^{q_2}\cdots |\Delta X_{T_{k_{d-l}}}|^{q_{d-l}} \Big)^2 \\
& \to 0 \quad \text{as} \quad m,m' \to \infty
\end{align*}
for $\mathbb{P}$-almost all $\omega$, since $\sum_{s\leq t} |\Delta X_s|^p$ is almost surely finite for any $p \geq 2$. Therefore we obtain, as $m,m' \to \infty$, 
\[
\tilde{\mathbb{P}} \Big(\sup_{t\in [0,t]}|Z(m)_t - Z(m')_t| > \epsilon \Big) = \int \mathbb{P}'\Big(\sup_{t\in [0,T]}|Z^{\omega}(m)_t - Z^{\omega}(m')_t|>\epsilon\Big) d\mathbb{P}(\omega) \to 0 
\]
by the dominated convergence theorem. The processes $Z(m)$ are c\`adl\`ag and contitute a Cauchy sequence in probability in the supremum norm. Hence they converge in probability to some $\tilde{\mathcal{F}}_t$-adapted c\`adl\`ag process $Z_t$. By the previous estimates we also obtain directly that
\begin{align}\label{l2con}
Z^{\omega} (m)_t \to Z_t(\omega, \cdot) \quad \text{in} \quad L^2(\Omega',\mathcal{F}',\mathbb{P}').
\end{align}
As a consequence it follows from \eqref{var} that 
\[
\int Z_t(\omega,\omega')^2 d\mathbb{P}'(\omega') =\sum_{s_1 \leq t} \Big(\sum_{s_2,\dots,s_{d-l}\leq t} \int_{[0,t]^l} \rho_{\partial_{l+1} H}(\sigma_{\mathbf{u}},\Delta X_{s_1},\Delta X_{s_2},\dots, \Delta X_{s_{d-l}}) d\mathbf{u} \Big)^2 \sigma_{s_1}^2 .
\]
Note that the multiple sum on the right hand side of the equation converges absolutely and hence does not depend on the choice of $(T_k)$. By \eqref{l2con} we obtain
\[
\mathbb{E}'(e^{i u Z^{\omega}(m)_t}) \to \mathbb{E}' (e^{i u Z_t( \omega, \cdot)}).
\]
Observe that for any centered square integrable random variable $U$ we have
\[
\Big|  \int (e^{i y U}-1) d\mathbb{P} \Big| \leq \mathbb{E} U^2 |y|^2 \quad \text{for all} \quad  y \in \mathbb{R}.
\]
Therefore the product in \eqref{char} converges absolutely as $m \to \infty$, and hence the characteristic function and thus the law of $Z_t(\omega,\cdot)$ do not depend on the choice of the sequence $(T_k)$. Lastly, observe that $R_k$ is $\mathcal{F}$-conditionally Gaussian. (In the case of a possibly discontinuous $\sigma$ as in Theorem \ref{stabvj} we need to require that $X$ and $\sigma$ do not jump at the same time to obtain such a property.) So we can conclude that $Z^{\omega}(m)_t$ is Gaussian, and $Z_t(\omega,\cdot)$ as a stochastic limit of Gaussian random variables is so as well.
\end{subsection}
\begin{subsection}{Uniform limit theory for continuous U-statistics}\label{ustg}
In this chapter we will give a proof of Proposition \ref{ustab}. Mainly we have to show that the sequence in \eqref{clte} is tight and that the finite dimensional distributions converge to the finite dimensional distributions of $\mathbb{U}_t$. For the convergence of the finite dimensional distributions we will generalize Proposition 4.3 in \cite{PSZ}. The basic idea in that work is to write the U-statistic as an integral with respect to the empirical distribution function
\[
F_n(t,x)=\frac{1}{n}\sum_{j=1}^{\left\lfloor nt\right\rfloor} \mathbbm{1}_{\left\{\alpha_j^n \leq x \right\}}.
\] 
In our setting we have
\[
\frac{1}{n^l}\sum_{\mathbf{i} \in \mathcal{B}_t^n (l) } G(\alpha_{\mathbf{i}}^n,\mathbf{y})= \int_{\mathbb{R}^l}G(\mathbf{x},\mathbf{y}) F_n(t,dx_1)\cdots F_n(t, dx_l).
\]
Of particular importance in \cite{PSZ} is the limit theory for the empirical process connected with $F_n$, which is given by
\[
\mathbb{G}_n(t,x)=\frac{1}{\sqrt{n}}\sum_{j=1}^{\left\lfloor nt\right\rfloor} \Big(\mathbbm{1}_{\left\{ \alpha_j^n \leq x\right\}} - \Phi_{\sigma_{\frac{j-1}{n}}}(x) \Big),
\]
where $\Phi_z$ is the cumulative distribution function of a standard normal random variable with variance $z^2$. As a slight generalization of \cite[Prop. 4.2]{PSZ} and by the same arguments as in \cite[Prop. 4.4.10]{JP} we obtain the joint convergence 
\[
(\mathbb{G}_n(t,x), (R_{-}(n,p),R_{+}(n,p))_{p \geq 1}) \stab (\mathbb{G}(t,x), (R_{p-},R_{p+})_{p\geq 1}).
\]
The stable convergence in law is to be understood as a process in $t$ and in the finite distribution sense in $x \in \mathbb{R}$. The limit is defined on an extension $(\tilde{\Omega}, \tilde{\mathcal{F}}, \tilde{\mathcal{P}})$ of the original probability space. $\mathbb{G}$ is $\mathcal{F}$-conditionally independent of $(\kappa_k, \psi_{k\pm})_{k\geq1}$ and $\mathcal{F}$-conditionally Gaussian and satisfies
\begin{align*} 
\tilde{\E}[\mathbb G&(t,x)|\mathcal F] = \int_0^t \overline{\Phi}_{\sigma_s} (x) dW_s,
\\
\tilde{\E}[\mathbb G&(t_1,x_1) \mathbb G(t_2,x_2) |\mathcal F] 
- \E'[\mathbb G(t_1,x_1) |\mathcal F]  \E'[\mathbb G(t_2,x_2) |\mathcal F] =  
\\
&\int_0^{t_1 \wedge t_2} \Phi_{\sigma_s} (x_1 \wedge x_2) - 
\Phi_{\sigma_s} (x_1) \Phi_{\sigma_s} (x_2) - \overline{\Phi}_{\sigma_s} (x_1) \overline{\Phi}_{\sigma_s} (x_2)ds,
\end{align*}
where $\overline{\Phi}_{z} (x) = \E[V\mathbbm{1}_{\{zV\leq x\}}]$ with $V\sim \mathcal N(0,1)$.

As in the proof of Prop. 4.3 in \cite{PSZ} we will use the decomposition 
\begin{align*}
\mathbb{U}_t^n(G,\mathbf{y})&= \sum_{k=1}^l \int_{\mathbb{R}^l} G(\mathbf{x},\mathbf{y}) \mathbb{G}_n(t, dx_k) \prod_{m=1}^{k-1} F_n(t,dx_m) \prod_{m=k+1}^l \bar{F_n}(t,dx_m) \\
&\quad \quad \quad \quad \quad \quad + \sqrt{n} \Big(\frac{1}{n^l} \sum_{\mathbf{j} \in \mathcal{B}_t^n (l)} \rho_G(\sigma_{(\mathbf{j}-1)/n},\mathbf{y}) -\int_{[0,t]^l} \rho_G (\sigma_{\mathbf{s}},\mathbf{y}) d\mathbf{s} \Big)=:\sum_{k=1}^l Z_k^n(G,\mathbf{y}) + R^n(\mathbf{y}),
\end{align*}
where
\[
\bar{F_n}(t,x)=\frac{1}{n} \sum_{j=1}^{\left\lfloor nt\right\rfloor} \Phi_{\sigma_{(j-1)/n}}(x).
\]
From \cite[Prop. 3.2]{PSZ} we know that both $F_n$ and $\bar{F_n}$ converge in probability to $F(t,x)=\int_0^t \Phi_{\sigma_s}(x) ds$ for fixed $t$ and $x$. If $G$ is symmetric and continuously differentiable in $\mathbf{x}$ with derivative of polynomial growth, \cite[Prop. 4.3]{PSZ} gives for fixed $\mathbf{y}$
\begin{align}\label{ustat}
\sum_{k=1}^l Z_k^n(G,\mathbf{y}) \stab \sum_{k=1}^l \int_{\mathbb{R}^l} G(\mathbf{x},\mathbf{y}) \mathbb{G}(t, dx_k) \prod_{m\neq k} F(t,dx_m)=:\sum_{k=1}^l Z_k(G,\mathbf{y}).
\end{align}
We remark that the proof of this result mainly relies on the following steps: First, use the convergence of $F_n$ and $\bar{F_n}$ and replace both by their limit $F$, which is differentiable in $x$. Then use the integration by parts formula for the Riemann-Stieltjes integral with respect to $\mathbb{G}_n(t,dx_k)$ plus the differentiability of $G$ in the $k$-th argument to obtain that $Z_k^n(G, \mathbf{y})$ is asymptotically the same as $-\int_{\mathbb{R}^l} \partial_k G(\mathbf{x},\mathbf{y}) \mathbb{G}_n(t, x_k) \prod_{m \neq k} F'(t,x_m) d\mathbf{x}$. Since one now only has convergence in finite dimensional distribution of $\mathbb{G}_n(t,\cdot)$ to $\mathbb{G}(t,\cdot)$, one uses a Riemann approximation of the integral with respect to $dx_k$ and takes limits afterwards. In the end do all the steps backwards.  

From the proof and the aforementioned joint convergence of $\mathbb{G}_n$ and $(R_{\pm}(n,p))_{p \geq 1}$ it is clear that we can slightly generalize \eqref{ustat} to 
\begin{align}\label{ustat2}
\Big((Z_k^n(G,\mathbf{y}))_{1\leq k \leq l}, (R_{-}(n,p),R_{+}(n,p))_{p \geq 1}) \Big) \stab \Big((Z_k(G,\mathbf{y}))_{1\leq k \leq l},(R_{p-},R_{p+})_{p\geq 1} \Big),
\end{align}
where the latter convergence holds in the finite distribution sense in $\mathbf{y}$ and also for non-symmetric, but still continuously differentiable functions $G$. A second consequence of the proof of \eqref{ustat} is that the mere convergence $Z_k^n(G,\mathbf{y}) \stab Z_k(G,\mathbf{y})$ only requires $G$ to be continuously differentiable in the $k$-th argument if $k$ is fixed. 

To show that \eqref{ustat2} holds in general under our assumptions let $\psi_{\epsilon} \in \mathcal{C}^{\infty}(\mathbb{R})$ ($\epsilon>0$) be functions with $0\leq \psi_{\epsilon} \leq 1$, $\psi_{\epsilon}(x) \equiv 1$ on $[-\epsilon/2,\epsilon/2]$,  $\psi_{\epsilon}(x) \equiv 0$ outside of $(-\epsilon,\epsilon)$, and $\left\|\psi_{\epsilon}'\right\| \leq K \epsilon^{-1}$ for some constant $K$, which is independent of $\epsilon$. Then the function $G(\mathbf{x})(1-\psi_{\epsilon}(x_k))$ is continuously differentiable in the $k$-th argument and hence it is enough to prove
\begin{align}\label{disc}
\lim_{\epsilon \to 0} \limsup_{n \to \infty} \mathbb{P}( \sup_{\mathbf{y}\in [-A,A]^{d-l}}|Z_k^n(G \psi_{\epsilon} ,\mathbf{y})| > \eta)=0 \\
\lim_{\epsilon \to 0} \mathbb{P}( \sup_{\mathbf{y}\in [-A,A]^{d-l}}|Z_k(G \psi_{\epsilon} ,\mathbf{y})| > \eta)=0 \label{disc2}
\end{align}
for all $\eta >0$ and $1\leq k\leq l$. For given $k$ the functions $\psi_{\epsilon}$ are to be evaluated at $x_k$. We show \eqref{disc} only for $k=l$. The other cases are easier since $\bar{F_n}$ is continuously differentiable in $x$ and the derivative is bounded by a continuous function with exponential decay at $\pm \infty$ since $\sigma$ is bounded away from $0$. 

For $k=l$, some $P \in \mathfrak{P}(1), Q\in \mathfrak{P}(l-1)$ and $x_l \neq 0$, we have 
\[
|\partial_l (G(\mathbf{x},\mathbf{y})\psi_{\epsilon}(x_l))|\leq K(1+|x_l|^{p_1-1})P(x_l)Q(x_1,\dots,x_{l-1})+K|x_1|^{p_1} \cdots |x_l|^{p_l}\epsilon^{-1}.
\]
Since $p_1-1 >-1$ the latter expression is integrable with respect to $x_l$ on compact intervals. Therefore the standard rules for the Riemann-Stieltjes integral and the monotonicity of $F_n$ in $x$ yield
\begin{align*}
\sup_{\mathbf{y}\in [-A,A]^{d-l}}|Z_l^n(G \psi_{\epsilon} ,\mathbf{y})|&=\sup_{\mathbf{y}\in [-A,A]^{d-l}} \Big|\int_{\mathbb{R}^l} G(\mathbf{x},\mathbf{y}) \psi_{\epsilon}(x_l)\mathbb{G}_n(t, dx_l) \prod_{m=1}^{l-1} F_n(t,dx_m) \Big| \\
& = \sup_{\mathbf{y}\in [-A,A]^{d-l}} \Big|\int_{\mathbb{R}^l} -\mathbb{G}_n(t, x_l)\partial_l (G(\mathbf{x},\mathbf{y}) \psi_{\epsilon}(x_l)) dx_l \prod_{m=1}^{l-1} F_n(t,dx_m) \Big|\\
&\leq  \int_{\mathbb{R}^{l-1}}\int_{-\epsilon}^{\epsilon}K |\mathbb{G}_n(t, x_l)|(1+|x_l|^{p_1-1})P(x_l)Q(x_1,\dots,x_{l-1}) dx_l \prod_{m=1}^{l-1} F_n(t,dx_m)  \\
& \quad + \int_{\mathbb{R}^{l-1}}\int_{-\epsilon}^{\epsilon}K |\mathbb{G}_n(t, x_l)||x_1|^{p_1} \cdots |x_l|^{p_l}\epsilon^{-1} dx_l \prod_{m=1}^{l-1} F_n(t,dx_m) \\
&=\int_{-\epsilon}^{\epsilon}K \Big(\frac{1}{n^{l-1}} \sum_{\mathbf{i} \in \mathcal{B}_t^n (l-1) } Q(\alpha_{\mathbf{i}}^n) \Big)|\mathbb{G}_n(t, x_l)|(1+|x_l|^{p_1-1})P(x_l) dx_l  \\
& \quad + \int_{-\epsilon}^{\epsilon}K \Big(\frac{1}{n^{l-1}} \sum_{\mathbf{i} \in \mathcal{B}_t^n (l-1) } |\alpha_{i_1}^n|^{p_1}\cdots |\alpha_{i_{l-1}}^n|^{p_{l-1}} \Big)|\mathbb{G}_n(t, x_l)| |x_l|^{p_l}\epsilon^{-1} dx_l.  
\end{align*} 
We have $\mathbb{E}|\alpha_i^n|^{q} \leq K$ uniformly in $i$ for every $q \geq 0$. From \cite[Lemma 4.1]{PSZ} it further follows that $\mathbb{E} |\mathbb{G}_n(t,x)|^q \leq K$ for all $q \geq 2$. Then we deduce from H\"older inequality
\[
\mathbb{E}\Big( \sup_{\mathbf{y}\in [-A,A]^{d-l}}|Z_l^n(G \psi_{\epsilon} ,\mathbf{y})| \Big) \leq K \int_{-\epsilon}^{\epsilon} (1+|x_l|^{p_1-1})P(x_l)+|x_l|^{p_l}\epsilon^{-1} dx_l ,
\] 
which converges to $0$ if we let $\epsilon \to 0$. We omit the proof of \eqref{disc2} since it follows by the same arguments. 

So far we have proven that \eqref{ustat2} holds under our assumptions on $G$. Furthermore, we can easily calculate the conditional covariance structure of the conditionally centered Gaussian process $\sum_{k=1}^l Z_k(G,\mathbf{y})$ by simply using that we know the covariance structure of $\mathbb{G}(t,x)$. We obtain the form in \eqref{cov}; for more details see \cite[sect. 5]{PSZ}.

Next we will show that 
\begin{align}\label{rest2}
\sup_{\mathbf{y}\in [-A,A]^{d-l}} |R^n(\mathbf{y})| \toop 0
\end{align}
as $n \to \infty$. Observe that $\rho_G(\mathbf{x},\mathbf{y})$ is $\mathcal{C}^{d+1}$ in the $\mathbf{x}$ argument. Therefore we get $R^n(\mathbf{y})\toop 0$ for any fixed $\mathbf{y}$ from \cite[sect. 7.3]{PSZ}. Further we can write
\begin{align}
R^n(\mathbf{y})=\sqrt{n}  \int_{[0,\left\lfloor nt\right\rfloor/n]^l} (\rho_G(\sigma_{\left\lfloor n\mathbf{s}\right\rfloor/n},\mathbf{y}) -\rho_G (\sigma_{\mathbf{s}},\mathbf{y})) d\mathbf{s} + \sqrt{n} \Big(\int_{[0,t]^l} \rho_G (\sigma_{\mathbf{s}},\mathbf{y}) d\mathbf{s}- \int_{[0,\left\lfloor nt\right\rfloor/n]^l} \rho_G (\sigma_{\mathbf{s}},\mathbf{y}) d\mathbf{s} \Big).
\end{align}
The latter term converges almost surely to $0$ and hence we can deduce \eqref{rest2} from the fact that \linebreak $\mathbb{E}|R^n(\mathbf{y})-R^n(\mathbf{y'})|\leq K \left\|\mathbf{y}-\mathbf{y'}\right\|$, which follows because $\rho_G(\mathbf{x},\mathbf{y})$ is continuously differentiable in $\mathbf{y}$ and $\mathbb{E}(\sqrt{n} |\sigma_{\left\lfloor nu\right\rfloor/n}-\sigma_u|) \leq K$ for all $u \in [0,t]$.  

Therefore we have proven the convergence of the finite dimensional distributions
\[
((\mathbb{U}_t^n(G,\mathbf{y}_i))_{i=1}^m, (R_{-}(n,p),R_{+}(n,p))_{p \geq 1})) \stab ((\mathbb{U}_t(G,\mathbf{y}_i))_{i=1}^m, (R_{p-},R_{p+})_{p\geq 1}).
\]
What remains to be shown in order to deduce Proposition \ref{ustab} is that the limiting process is indeed continuous and that the sequences $Z_k^n(G,\cdot)$ ($1\leq k\leq l$) are tight. For the continuity of the limit observe that
\begin{align*}
\mathbb{E}[|\mathbb{U}_t(G,\mathbf{y})&-\mathbb{U}_t(G,\mathbf{y}')|^2|\mathcal{F}] \\
&=\int_0^t \Big(\int_{\mathbb{R}} \Big(\sum_{i=1}^l (f_i(u,\mathbf{y})-f_i(u,\mathbf{y'})) \Big)^2 \phi_{\sigma_s}(u) du -\Big(\sum_{i=1}^l \int_{\mathbb{R}}(f_i(u,\mathbf{y})-f_i(u,\mathbf{y'})) \phi_{\sigma_s}(u) du \Big)^2 ds\Big).
\end{align*}
Here we can use the differentiability assumptions and the boundedness of $\sigma$ and $\sigma^{-1}$ to obtain
\[
\mathbb{E}[|\mathbb{U}_t(G,\mathbf{y})-\mathbb{U}_t(G,\mathbf{y}')|^2]=\mathbb{E}[\mathbb{E}[|\mathbb{U}_t(G,\mathbf{y})-\mathbb{U}_t(G,\mathbf{y}')|^2|\mathcal{F}]]\leq K \left\|\mathbf{y}-\mathbf{y}'\right\|^2.
\]
Since $\mathbb{U}_t(G,\cdot)$ is $\mathcal{F}$-conditionally Gaussian we immediately get 
\[
\mathbb{E}[|\mathbb{U}_t(G,\mathbf{y})-\mathbb{U}_t(G,\mathbf{y}')|^p]\leq K_p \left\|\mathbf{y}-\mathbf{y}'\right\|^p
\]
for any even $p\geq 2$. In particular, this implies that there exists a continuous version of the multiparameter process $\mathbb{U}_t(G,\cdot)$ (see \cite[Theorem 2.5.1]{DK}). 

The last thing we need to show is tightness. A tightness criterion for multiparameter processes can be found in \cite{BW}. Basically we have to control the size of the increments of the process on blocks (and on lower boundaries of blocks, which works in the same way). By a block we mean a set $B \subset [-A,A]^{d-l}$ of the form $B=(y_1,y_1'] \times \dots \times (y_{d-l},y_{d-l}']$, where $y_i < y_i'$. An increment of a process $Z$ defined on $[-A,A]^{d-l}$ on such a block is defined by
\[
\Delta_B (Z) :=\sum_{i_1,\dots, i_{d-l}=0}^1 (-1)^{d-l-\sum_{j} i_j} Z(y_1+ i_1 (y_1'-y_1),\dots, y_{d-l}+i_{d-l}(y_{d-l}'-y_{d-l})).
\]  
We remark that if $Z$ is sufficiently differentiable, then 
\[
\Delta_B(Z)=\partial_1 \cdots \partial_{d-l} Z(\mathbf{\xi}) (y_1'-y_1)\cdot \ldots \cdot (y_{d-l}'-y_{d-l}) 
\]
for some $\mathbf{\xi} \in B$. We will now show tightness for the process $Z_l^n(G,\mathbf{y})$. According to \cite{BW} it is enough to show 
\[
\mathbb{E}[|\Delta_B(Z_l^n(G,\cdot))|^2]\leq K (y_1'-y_1)^2\cdot \ldots \cdot (y_{d-l}'-y_{d-l})^2
\]
in order to obtain tightness. As before we use the standard properties of the Riemann-Stieltjes integral to deduce
\begin{align*}
\mathbb{E}[|\Delta_B(Z_l^n(G,\cdot))|^2]&=\mathbb{E}\Big[\Big(\int_{\mathbb{R}^l} \Delta_B (G(\mathbf{x},\cdot)) \mathbb{G}_n(t,dx_l) \prod_{k=1}^{l-1} F_n(t,dx_k)\Big)^2 \Big]\\
&=\mathbb{E}\Big[\Big(\int_{\mathbb{R}^l} \Delta_B (\partial_l G(\mathbf{x},\cdot)) \mathbb{G}_n(t,x_l)dx_l \prod_{k=1}^{l-1} F_n(t,dx_k) \Big)^2\Big] \\
&=\mathbb{E}\Big[\Big(\int_{\mathbb{R}^l} \partial_l \partial_{l+1} \cdots \partial_d  G(\mathbf{x},\mathbf{\xi}) \mathbb{G}_n(t,x_l)dx_l \prod_{k=1}^{l-1} F_n(t,dx_k) \Big)^2\Big] \prod_{i=1}^l (y_i - y_i')^2
\end{align*}
for some $\mathbf{\xi} \in B$. As it is shown in \cite{PSZ} there exists a continuous function $\gamma :\mathbb{R} \to \mathbb{R}$ with exponential decay at $\pm \infty$ such that $\mathbb{E}[\mathbb{G}_n (t, x)^4]\leq \gamma(x)$. Using the growth assumptions on $L$ we further know that there exist $P \in \mathfrak{P}(1)$ and $Q \in \mathfrak{P}(l-1)$ such that
\[
|\partial_l \partial_{l+1} \cdots \partial_d  G(\mathbf{x},\mathbf{\xi})| \leq K(1+|x_l|^{p_l-1})P(x_l)Q(x_1,\dots, x_{l-1})
\]
and hence
\begin{align*}
&\mathbb{E}\Big[\Big(\int_{\mathbb{R}^l} \partial_l \partial_{l+1} \cdots \partial_d  G(\mathbf{x},\mathbf{\xi}) \mathbb{G}_n(t,x_l)dx_l \prod_{k=1}^{l-1} F_n(t,dx_k) \Big)^2\Big] \\
\leq & K\mathbb{E}\Big[\int_{\mathbb{R}^2} \Big(\frac{1}{n^{l-1}} \sum_{\mathbf{i} \in \mathcal{B}_t^n (l-1) } Q(\alpha_{\mathbf{i}}^n) \Big)^2 (1+|x_l|^{p_l-1})(1+|x_l'|^{p_l-1})P(x_l) P(x_l')|\mathbb{G}_n(t, x_l) \mathbb{G}_n(t,x_l')| dx_l dx_l' \Big] \leq K
\end{align*}
by Fubini, the Cauchy-Schwarz inequality, and the aforementioned properties of $\mathbb{G}_n(t,x)$. The proof for the tightness of $Z_k^n(G,\mathbf{y})$ $(1\leq k \leq l-1)$ is similar and therefore omitted. 
\end{subsection}

\begin{subsection}{Proofs of some technical results}
\textit{Proof of Proposition \ref{ap}:} \\
i) \noindent For $j>0$ consider the terms $\zeta_{k,j}^n (m)$ and $\tilde{\zeta}_{k,j}^n(m)$, which appear in decomposition \eqref{dcom}. Since $X$ is bounded and $\mathcal{P}_t^n(m)$ a finite set, we have the estimate
\[
\max(|\zeta_{k,j}^n (m)| , |\tilde{\zeta}_{k,j}^n(m)|)\leq K(m)\sqrt{n}n^{-j} \sum_{\mathbf{i} \in \mathcal{B}_t^n(l-k)} |\Delta_{i_1}^n X(m)|^p \cdots |\Delta_{i_{l-k}}^n X(m)|^p.
\]
By \eqref{ord1} we therefore obtain 
\[
\mathbb{E}(\mathbbm{1}_{\Omega_n (m)}(|\zeta_{k,j}^n (m)| +|\tilde{\zeta}_{k,j}^n(m)|)) \to 0 \quad \text{as} \quad n\to \infty.
\]
In the case $k>0$ we have 
\[
|\tilde{\zeta}_{k,j}^n(m)| \leq K(m) \sqrt{n}n^{-j-k\frac{p}{2}}  \sum_{\mathbf{p}\in \mathcal{P}_t^n(m)^k} |R(n,p_1)|^p\cdots |R(n,p_k)|^p  \sum_{\mathbf{i} \in \mathcal{B}_t^n(l-k)} |\Delta_{i_1}^n X(m)|^p \cdots |\Delta_{i_{l-k}}^n X(m)|^p.
\]
Since $(R(n,p))$ is bounded in probability as a sequence in $n$, we can deduce
\[
\mathbbm{1}_{\Omega_n (m)}|\tilde{\zeta}_{k,j}^n(m)| \toop 0 \quad \text{as} \quad n \to \infty.
\]
Furthermore, in the case $j=k=0$, we have $\zeta_{0,0}^n(m) =\tilde{\zeta}_{0,0}^n(m)$.

(ii) At last we have to show the convergence
\[
\lim_{m\to \infty} \limsup_{n\to \infty} \mathbb{P}(\mathbbm{1}_{\Omega_n(m)} |\sum_{k=1}^{l-1}  \big(\zeta_{k,0}^n(m) -  \zeta_k^n(m)\big)| >\eta) =0 \quad \text{for all} \quad \eta > 0. 
\]
First we will show in a number of steps that we can replace $\Delta X_{S_{\mathbf{p}}} + \frac{1}{\sqrt{n}} R(n,\mathbf{p})$ by $\Delta X_{S_{\mathbf{p}}}$ in $\zeta_{k,0}^n(m)$ without changing the asymptotic behaviour. Fix $k\in \left\{1,\dots, l-1 \right\}$. We start with
\begin{align*}
&\Bigg|\binom{l}{k}^{-1} \zeta_{k,0}^n(m)-\frac{\sqrt{n}}{n^{d-l}} \sum_{\textbf{p}\in \mathcal{P}_t^n(m)^{k-1} \atop p_k \in \mathcal{P}_t^n(m)} \sideset{}{'}\sum_{\textbf{i} \in \mathcal{B}_t^n (d-k) }  H\Big(\Delta X_{S_{\textbf{p}}}+\frac{1}{\sqrt{n}} R(n,\textbf{p}),\Delta X_{S_{p_k}}, \Delta_{\textbf{i}}^n X(m)\Big) \Bigg|  \\
=&\Bigg| \frac{\sqrt{n}}{n^{d-l}} \sum_{\textbf{p}\in \mathcal{P}_t^n(m)^{k-1} \atop p_k \in \mathcal{P}_t^n(m)}\sideset{}{'} \sum_{\textbf{i} \in \mathcal{B}_t^n (d-k) } \int_{0}^{\frac{R(n,p_k)}{\sqrt{n}}} \partial_k H\Big(\Delta X_{S_{\textbf{p}}}+\frac{1}{\sqrt{n}} R(n,\textbf{p}),\Delta X_{S_{p_k}}+u, \Delta_{\textbf{i}}^n X(m)\Big) du \Bigg| \\
\leq & K \sum_{\textbf{p}\in \mathcal{P}_t(m)^{k-1} \atop p_k \in \mathcal{P}_t(m)}  |R(n,p_k)| \sup_{|u|,|v|\leq \frac{|R(n,p_k)|}{\sqrt{n}}} (|\Delta X_{S_{p_k}}+u|^p+|\Delta X_{S_{p_k}}+v|^{p-1})\prod_{r=1}^{k-1} \Big|\Delta X_{S_{p_r}}+\frac{R(n,p_r)}{\sqrt{n}}\Big|^p  \\
&  \quad \quad \quad \quad \quad \quad \quad \times\sum_{\textbf{i} \in \mathcal{B}_t^n (l-k) }\prod_{j=1}^{l-k} |\Delta_{i_j}^n X(m)|^p\\
&=: K \Phi_1^n(m) \times \Phi_2^n (m).
\end{align*}
The first factor $\Phi_1^n(m)$ converges, as $n \to \infty$, stably in law towards 
\[
\Phi_1(m)=\sum_{\textbf{p}\in \mathcal{P}_t(m)^{k-1} \atop p_k \in \mathcal{P}_t(m) }  |R_{p_k}|  (|\Delta X_{S_{p_k}}|^p+|\Delta X_{S_{p_k}}|^{p-1})\prod_{r=1}^{k-1} |\Delta X_{S_{p_r}}|^p.  
\]
By the Portmanteau theorem we obtain 
\[
\limsup_{n\to \infty} \mathbb{P}(|\Phi_1^n(m)| \geq M) \leq  \mathbb{\tilde{P}}(|\Phi_1(m)| \geq M) \quad \text{for all} \quad M \in \mathbb{R_{+}},
\]
 whereas, as $m \to \infty$, 
\[
\Phi_1(m) \stackrel{\tilde{\mathbb{P}}}{\longrightarrow} \Bigg( \sum_{s\leq t} |\Delta X_s |^p \Bigg)^{k-1} \sum_{p_k \in \mathcal{P}_t}  R_{p_k}  (|\Delta X_{S_{p_k}}|^p+|\Delta X_{S_{p_k}}|^{p-1}).
\]
So it follows that
\[
\lim_{M \to \infty} \limsup_{m\to \infty} \limsup_{n\to \infty} \mathbb{P}(|\Phi_1^n(m)| \geq M)=0.
\]
Furthermore
\[
\lim_{m\to \infty} \limsup_{n\to \infty} \mathbb{E}(\mathbbm{1}_{\Omega_n(m)} \Phi_2^n(m)) \leq  \lim_{m\to \infty} \limsup_{n\to \infty} \frac{K}{m^{(l-k)(p-2)}}\mathbb{E}\bigg(\sum_{\textbf{i} \in \mathcal{B}_t^n (l-k) }\prod_{j=1}^{l-k} |\Delta_{i_j}^n X(m)|^2 \bigg)=0
\]
by Lemma  \ref{lem1}. We finally obtain 
\[
\lim_{m\to \infty} \limsup_{n\to \infty} \mathbb{P}\big(\mathbbm{1}_{\Omega_n(m)} |\Phi_1^n(m)  \Phi_2^n(m)|> \eta \big)=0 \quad \text{for all} \quad \eta >0.
\] 
Doing these steps successively in the first $k-1$ components as well, we obtain
\[
\lim_{m\to \infty} \limsup_{n\to \infty} \mathbb{P}\big(\mathbbm{1}_{\Omega_n(m)} \Big|\binom{l}{k}^{-1}\zeta_{k,0}^n(m)  -\theta_k^n(m)\Big|> \eta \big)=0 \quad \text{for all} \quad \eta >0
\]
 with
\[
\theta_k^n(m):= \frac{\sqrt{n}}{n^{d-l}}\sum_{\textbf{p}\in \mathcal{P}_t(m)^{k}} \sum_{\textbf{i} \in \mathcal{B}_t^n (d-k) }  H\Big(\Delta X_{S_{\textbf{p}}}, \Delta_{\textbf{i}}^n X(m)\Big).
\]
By the same arguments as in the proof of the convergence $\mathbbm{1}_{\Omega_n(m)} \Psi_1^n(m) \toop 0$ in section \ref{cltjc} we see that we can replace the last $d-l$ variables of $H$ in $\mathbbm{1}_{\Omega_n(m)}\theta_k^n(m)$ by $0$ without changing the limit. So we can restrict ourselves without loss of generality to the case $l=d$ now and have to prove 
\begin{align}\label{conv}
\lim_{m \to \infty} \limsup_{n \to \infty} \mathbb{P}(\mathbbm{1}_{\Omega_n(m)} |\Theta_k^n(m)| > \eta) =0
\end{align}
with
\begin{align*}
\Theta_k^n(m) :=\sqrt{n} \sum_{\textbf{p}\in \mathcal{P}_t(m)^{k}} \Big(\sum_{\textbf{i} \in \mathcal{B}_t^n (d-k) }  H\Big(\Delta X_{S_{\textbf{p}}}, \Delta_{\textbf{i}}^n X(m)\Big) -\sum_{\mathbf{s} \in (0,\frac{\left\lfloor nt\right\rfloor}{n}]^{d-k}}  H\Big(\Delta X_{S_{\textbf{p}}}, \Delta X(m)_{\mathbf{s}}\Big)\Big).
\end{align*}
Since 
\[
\sum_{q \in \mathcal{P}_t(m)} |\Delta X_{S_q}|^p \leq \sum_{s\leq t} |\Delta X_s|^p 
\]
is bounded in probability, we can adopt exactly the same method as in the proof of $\mathbbm{1}_{\Omega_n(m)} \Psi_2^n(m) \toop 0$ to show \eqref{conv}, which finishes the proof of Proposition \ref{ap}. \qed

\noindent \textit{Proof of Proposition \ref{appal}:} \\
We will only show that we can replace $\sqrt{n}\Delta_i^n X^c$ by $\alpha_i^n$ in the first argument, i.e.\ the convergence
\begin{align}\label{appal2}
\zeta_n:=\frac{\sqrt{n}}{n^l}\sum_{k=1}^{\left\lfloor nt\right\rfloor}\sum_{\textbf{i}\in \mathcal{B}_t^n (l-1)}\sum_{\textbf{j}\in \mathcal{B}_t^n (d-l)}\Big( H(\sqrt{n}\Delta_k^n X^c,\sqrt{n}\Delta_{\mathbf{i}}^n X^c,\Delta_{\mathbf{j}}^n X) -H(\alpha_k^n,\sqrt{n}\Delta_{\mathbf{i}}^n X^c,\Delta_{\mathbf{j}}^n X) \Big) \toop 0.
\end{align}
All the other replacements follow in the same manner. Define the function $g:\mathbb{R}^d \to \mathbb{R}$ by $g(w,\mathbf{x},\mathbf{y})=|w|^{p_1} L(w,\mathbf{x},\mathbf{y})$. In a first step we will show that, for fixed $M>0$, we have
\begin{align}\label{supcon}
\frac{1}{\sqrt{n}}\sup_{\left\|\mathbf{z}\right\| \leq M} \sum_{k=1}^{\left\lfloor nt\right\rfloor}\big(g(\sqrt{n}\Delta_k^n X^c, \mathbf{z})-g(\alpha_k^n, \mathbf{z})\big) \toop 0,   
\end{align}
where $\mathbf{z}=(\mathbf{x},\mathbf{y})\in \mathbb{R}^{l-1}\times \mathbb{R}^{d-l}$. Note that our growth assumptions on $L$ imply the existence of constants $h,h',h'' \geq 0$ such that
\begin{align}
w \neq 0 \Longrightarrow &|\partial_1 g(w, \mathbf{x}, \mathbf{y})| \leq K u(\mathbf{y}) (1+\left\|(w,\mathbf{x})\right\|^h)\Big(1+|w|^{p_1-1} \Big) \label{incr1}\\
w \neq 0, |z| \leq |w|/2 \Longrightarrow &|\partial_1 g(w+z, \mathbf{x}, \mathbf{y})-\partial_1 g(w, \mathbf{x}, \mathbf{y})| \leq Ku(\mathbf{y})|z|(1+\left\|(w,\mathbf{x})\right\|^{h'}+|z|^{h'}) \Big( 1+|w|^{p_1-2}\Big) \label{incr2}\\
&|g(w+z, \mathbf{x}, \mathbf{y})-g(w, \mathbf{x}, \mathbf{y})| \leq Ku(\mathbf{y})(1+\left\|(w,\mathbf{x})\right\|^{h''})|z|^{p_1} \label{incr3}
\end{align}
The first inequality is trivial, the second one follows by using the mean value theorem, and the last one can be deduced by the same arguments as in the derivation of \eqref{ineq}.
In particular, for fixed $\mathbf{x}, \mathbf{y}$ all assumptions of \cite[Theorem 5.3.6]{JP} are fulfilled and hence
\begin{align*}
\frac{1}{\sqrt{n}}\max_{\mathbf{z} \in K_m(M)} \sum_{k=1}^{\left\lfloor nt\right\rfloor}\big(g(\sqrt{n}\Delta_k^n X^c, \mathbf{z})-g(\alpha_k^n, \mathbf{z})\big) \toop 0,   
\end{align*}
where $K_m(M)$ is defined to be a finite subset of $[-M,M]^{d-1}$ such that for each $\mathbf{z} \in [-M,M]^{d-l}$ there exists $\mathbf{z}'\in K_m(M)$ with $\left\|\mathbf{z}-\mathbf{z}'\right\| \leq 1/m$. In order to show \eqref{supcon} it is therefore enough to prove
\[
\frac{1}{\sqrt{n}}\sup_{\left\|(\mathbf{z_1},\mathbf{z_2})\right\| \leq M \atop \left\|\mathbf{z}_1-\mathbf{z}_2\right\|\leq 1/m} \Big|\sum_{k=1}^{\left\lfloor nt\right\rfloor}\big(g(\sqrt{n}\Delta_k^n X^c, \mathbf{z}_1)-g(\alpha_k^n, \mathbf{z}_1)-(g(\sqrt{n}\Delta_k^n X^c, \mathbf{z}_2)-g(\alpha_k^n, \mathbf{z}_2))\big)\Big| \toop 0
\]
if we first let $n$ and then $m$ go to infinity. 

Now, let $\theta_k^n =\sqrt{n} \Delta_k^n X^c - \alpha_k^n$ and $B_k^n=\left\{|\theta_k^n| \leq |\alpha_k^n|/2 \right\}$. Clearly, $g$ is differentiable in the last $d-1$ arguments and on $B_k^n$ we can also apply the mean value theorem in the first argument. We therefore get
\begin{align*}
\mathbbm{1}_{B_k^n}\big(g(\sqrt{n}\Delta_k^n X^c, \mathbf{z}_1)-g(\alpha_k^n, \mathbf{z}_1)-(g(\sqrt{n}\Delta_k^n X^c, \mathbf{z}_2)-g(\alpha_k^n, \mathbf{z}_2))\big) =\sum_{j=2}^d \mathbbm{1}_{B_k^n}\partial_1 \partial_j g(\chi_{j,k}^n, \mathbf{\xi}_{j,k}^n)(z_2^{(j)}-z_1^{(j)})\theta_k^n,
\end{align*}
where $\chi_{j,k}^n$ is between $\sqrt{n}\Delta_k^n X^c$ and $\alpha_k^n$ and $\xi_{j,k}^n$ is between $\mathbf{z}_1$ and $\mathbf{z}_2$. $z_i^{(j)}$ stands for the $j$-th component of $\mathbf{z}_i$. We have $|\partial_1 \partial_j g(w,\mathbf{z})|\leq p_1|w|^{p_1-1} |\partial_j L(w, \mathbf{z})|+ |w|^{p_1} |\partial_1 \partial_j L(w,\mathbf{z})|$ and therefore the growth conditions on $L$ imply that there exists $q \geq 0$ such that 
\[
|\partial_1 \partial_j g(w,\mathbf{z})| \leq Ku(\mathbf{y}) (1+|w|^{p_1-1})(1+\left\|(w,\mathbf{x})\right\|^q).
\]
On $B_k^n$ we have $|\chi_{j,k}^n| \leq \frac{3}{2} |\alpha_k^n|$. From $\left\|\mathbf{z}\right\| \leq M$ we find  
\begin{align*}
&\frac{1}{\sqrt{n}}\mathbb{E}\Bigg(\sup_{\left\|(\mathbf{z_1},\mathbf{z_2})\right\| \leq M \atop \left\|\mathbf{z}_1-\mathbf{z}_2\right\|\leq 1/m} \Big|\sum_{k=1}^{\left\lfloor nt\right\rfloor}\mathbbm{1}_{B_k^n}\big(g(\sqrt{n}\Delta_k^n X^c, \mathbf{z}_1)-g(\alpha_k^n, \mathbf{z}_1)-(g(\sqrt{n}\Delta_k^n X^c, \mathbf{z}_2)-g(\alpha_k^n, \mathbf{z}_2))\big)\Big|\Bigg) \\
\leq &  \frac{K(M)}{\sqrt{n}m}\sum_{k=1}^{\left\lfloor nt\right\rfloor}\mathbb{E} \Big( (1+|\alpha_k^n|^{p_1-1})(1+|\alpha_k^n|^{q}+|\sqrt{n}\Delta_k^n X^c|^q) |\theta_k^n|\Big).
\end{align*}
By Burkholder inequality we know that $\mathbb{E}\big((1+|\alpha_k^n|^{q}+|\sqrt{n}\Delta_k^n X^c|^q)^u\big) \leq K$ for all $u\geq 0$. Since $\sigma$ is a continuous semimartingale we further have $\mathbb{E} (|\theta_k^n|^u) \leq Kn^{-u/2}$ for $u \geq 1$. Finally, because $\sigma$ is bounded away from $0$, we also have $\mathbb{E}\big( (|\alpha_k^n|^{p_1-1})^u\big) \leq K$ for all $u \geq 0$ with $u(1-p_1)<1$. Using this results in combination with H\"older inequality we obtain
\[
\frac{K(M)}{\sqrt{n}m}\sum_{k=1}^{\left\lfloor nt\right\rfloor}\mathbb{E} \Big( (1+|\alpha_k^n|^{p_1-1})(1+|\alpha_k^n|^{q}+|\sqrt{n}\Delta_k^n X^c|^q) |\theta_k^n|\Big) \leq \frac{K(M)}{m},
\]
which converges to $0$ as $m \to \infty$.

Now we focus on $(B_k^n)^C$. Let $2\leq j\leq d$. Observe that, similarly to \eqref{ineq}, by distinguishing the cases $|z|\leq 1$ and $|z|>1$, we find that 
\[
|\partial_j L(w+z, \mathbf{x}, \mathbf{y})- \partial_j L(w,\mathbf{x},\mathbf{y})| \leq K(1+|w|^{\gamma_j+\gamma_{1j}})|z|^{\gamma_j}. 
\]
We used here that $\left\|(\mathbf{x},\mathbf{y})\right\|$ is bounded and the simple inequality $1+a+b \leq 2(1+a)b$ for all $a\geq 0, b\geq 1$. From this we get
\begin{align*}
&|\partial_j g (w+z, \mathbf{x},\mathbf{y})-\partial_j g(w,\mathbf{x},\mathbf{y})| \\
\leq  &\Big||w+z|^{p_1}-|w|^{p_1}\Big| |\partial_j L(w+z,\mathbf{x},\mathbf{y})|+|w|^{p_1}|\partial_j L(w+z, \mathbf{x}, \mathbf{y})- \partial_j L(w,\mathbf{x},\mathbf{y})| \\
\leq & K (1+|w|^q)(|z|^{\gamma_j + p_1}+|z|^{\gamma_j})
\end{align*}
 for some $q\geq 0$. Recall that $\gamma_j <1$ and $\gamma_j + p_1 <1$ by assumption. For some $\xi_j^n$ between $z_1^{(j)}$ and $z_2^{(j)}$ we therefore have
\begin{align*}
& \frac{1}{\sqrt{n}}\mathbb{E}\Bigg(\sup_{\left\|(\mathbf{z_1},\mathbf{z_2})\right\| \leq M \atop \left\|\mathbf{z}_1-\mathbf{z}_2\right\|\leq 1/m} \Big|\sum_{k=1}^{\left\lfloor nt\right\rfloor}\mathbbm{1}_{(B_k^n)^C}\big(g(\sqrt{n}\Delta_k^n X^c, \mathbf{z}_1)-g(\alpha_k^n, \mathbf{z}_1)-(g(\sqrt{n}\Delta_k^n X^c, \mathbf{z}_2)-g(\alpha_k^n, \mathbf{z}_2))\big)\Big|\Bigg) \\
=& \frac{1}{\sqrt{n}}\mathbb{E}\Bigg(\sup_{\left\|(\mathbf{z_1},\mathbf{z_2})\right\| \leq M \atop \left\|\mathbf{z}_1-\mathbf{z}_2\right\|\leq 1/m} \Big|\sum_{k=1}^{\left\lfloor nt\right\rfloor}\sum_{j=2}^d \mathbbm{1}_{(B_k^n)^C}\big(\partial_j g(\sqrt{n}\Delta_k^n X^c, \xi_j^n)-\partial_j g(\alpha_k^n, \xi_j^n)\big)(z_2^{(j)}-z_1^{(j)})\Big|\Bigg) \\
\leq &\frac{K(M)}{\sqrt{n}m} \sum_{k=1}^{\left\lfloor nt\right\rfloor}\mathbb{E}\Big(\mathbbm{1}_{(B_k^n)^C}(1+|\alpha_k^n|^{q}+|\sqrt{n}\Delta_k^n X^c|^q) (|\theta_k^n|^{\gamma_1}+|\theta_k^n|^{\gamma_j +p_1})  \Big) \\
\leq &\frac{K(M)}{\sqrt{n}m} \sum_{k=1}^{\left\lfloor nt\right\rfloor}\mathbb{E}\Big(\mathbbm{1}_{(B_k^n)^C}(1+|\alpha_k^n|^{q}+|\sqrt{n}\Delta_k^n X^c|^q) \Big(\frac{|\theta_k^n|}{|\alpha_k^n|^{1-\gamma_1}}+\frac{|\theta_k^n|}{|\alpha_k^n|^{1-(\gamma_j + p_1)}} \Big) \Big) \leq \frac{K(M)}{m}
\end{align*}
by the same arguments as before, and hence \eqref{supcon} holds. For any $M>2A$ we therefore have (with $\mathbf{q}=(q_1,\dots,q_{d-l})$)
\begin{align*}
&\Big|\frac{\sqrt{n}}{n^l}\sum_{k=1}^{\left\lfloor nt\right\rfloor}\sum_{\textbf{i}\in \mathcal{B}_t^n (l-1)}\sum_{\textbf{j}\in \mathcal{B}_t^n (d-l)}\mathbbm{1}_{\left\{\left\|\sqrt{n}\Delta_{\mathbf{i}}^n X^c\right\| \leq M \right\}}\Big( H(\sqrt{n}\Delta_k^n X^c,\sqrt{n}\Delta_{\mathbf{i}}^n X^c,\Delta_{\mathbf{j}}^n X) -H(\alpha_k^n,\sqrt{n}\Delta_{\mathbf{i}}^n X^c,\Delta_{\mathbf{j}}^n X) \Big)\Big|\\
\leq &  \Big( \frac{1}{n^{l-1}} \sum_{\textbf{i}\in \mathcal{B}_t^n (l-1)}\sum_{\textbf{j}\in \mathcal{B}_t^n (d-l)}|\sqrt{n}\Delta_{i_1}X^c|^{p_2}\cdots |\sqrt{n}\Delta_{i_{l-1}}X^c|^{p_l} |\Delta_{\mathbf{j}}^n X|^{\mathbf{q}}\Big) \Big| \frac{1}{\sqrt{n}}\sup_{\left\|\mathbf{z}\right\| \leq M} \sum_{k=1}^{\left\lfloor nt\right\rfloor}\big(g(\sqrt{n}\Delta_k^n X^c, \mathbf{z})-g(\alpha_k^n, \mathbf{z})\big) \Big|
\end{align*}
The first factor converges in probability to some finite limit, and hence the whole expression converges to $0$ by \eqref{supcon}. In order to show \eqref{appal2} we are therefore left with proving
\[
\frac{\sqrt{n}}{n^l}\sum_{k=1}^{\left\lfloor nt\right\rfloor}\sum_{\textbf{i}\in \mathcal{B}_t^n (l-1)}\sum_{\textbf{j}\in \mathcal{B}_t^n (d-l)}\mathbbm{1}_{\left\{\left\|\sqrt{n}\Delta_{\mathbf{i}}^n X^c\right\| > M \right\}}\Big( H(\sqrt{n}\Delta_k^n X^c,\sqrt{n}\Delta_{\mathbf{i}}^n X^c,\Delta_{\mathbf{j}}^n X) -H(\alpha_k^n,\sqrt{n}\Delta_{\mathbf{i}}^n X^c,\Delta_{\mathbf{j}}^n X) \Big) \toop 0,
\] 
if we first let $n$ and then $M$ go to infinity. As before we will distinguish between the cases that we are on the set $B_k^n$ and on $(B_k^n)^C$. Let $\tilde{\mathbf{p}}=(p_2,\dots,p_l)$. With the mean value theorem and the growth properties of $\partial_1 g$ from \eqref{incr1} we obtain for all $M\geq1$: 
\begin{align*}
&\Big| \frac{\sqrt{n}}{n^l}\sum_{k=1}^{\left\lfloor nt\right\rfloor}\sum_{\textbf{i}\in \mathcal{B}_t^n (l-1)}\sum_{\textbf{j}\in \mathcal{B}_t^n (d-l)}\mathbbm{1}_{\left\{\left\|\sqrt{n}\Delta_{\mathbf{i}}^n X^c\right\| > M \right\}}\mathbbm{1}_{B_k^n}\Big( H(\sqrt{n}\Delta_k^n X^c,\sqrt{n}\Delta_{\mathbf{i}}^n X^c,\Delta_{\mathbf{j}}^n X) -H(\alpha_k^n,\sqrt{n}\Delta_{\mathbf{i}}^n X^c,\Delta_{\mathbf{j}}^n X) \Big) \Big| \\
\leq &K\Big(\sum_{\textbf{j}\in \mathcal{B}_t^n (d-l)} |\Delta_{\mathbf{j}}^n X|^{\mathbf{q}} \Big) \frac{\sqrt{n}}{n^l}\sum_{k=1}^{\left\lfloor nt\right\rfloor}\sum_{\textbf{i}\in \mathcal{B}_t^n (l-1)} |\sqrt{n}\Delta_{\mathbf{i}}^n X^c|^{\tilde{\mathbf{p}}} \mathbbm{1}_{\left\{\left\|\sqrt{n}\Delta_{\mathbf{i}}^n X^c\right\| > M \right\}}\\
& \quad \quad \quad \quad \quad \quad \quad \quad \quad \quad \quad \quad \quad \quad \quad \quad \quad \quad \times (1+|\alpha_k^n|^h+|\sqrt{n}\Delta_k^n X^c|^h+ \left\|\sqrt{n}\Delta_{\mathbf{i}}^n X^c\right\|^h)\big(1+|\alpha_k^n|^{p_1-1} \big)|\theta_k^n|\\
& \leq \Big(K\sum_{\textbf{j}\in \mathcal{B}_t^n (d-l)} |\Delta_{\mathbf{j}}^n X|^{\mathbf{q}} \Big) \Big(\frac{1}{n^{l-1}} \sum_{\textbf{i}\in \mathcal{B}_t^n (l-1)}\mathbbm{1}_{\left\{\left\|\sqrt{n}\Delta_{\mathbf{i}}^n X^c\right\| > M \right\}} |\sqrt{n}\Delta_{\mathbf{i}}^n X^c|^{\tilde{\mathbf{p}}}\left\|\sqrt{n}\Delta_{\mathbf{i}}^n X^c\right\|^h\Big)\\
& \quad \quad \quad \quad \quad \quad \quad \quad \quad \quad \quad \quad \quad \quad \quad \quad \quad \quad \times \Big(\frac{1}{\sqrt{n}}\sum_{k=1}^{\left\lfloor nt\right\rfloor} (1+|\alpha_k^n|^h+|\sqrt{n}\Delta_k^n X^c|^h)\big(1+|\alpha_k^n|^{p_1-1} \big)|\theta_k^n|\Big)\\
&=:A_n  B_n(M)  C_n,
\end{align*}
where we used $M\geq 1$ and $1+a+b \leq 2(1+a)b$ for the final inequality again. As before, we deduce that $A_n$ is bounded in probability and $\mathbb{E}(C_n)\leq K$. We also have $\mathbb{E}(B_n(M))\leq K/M$ and hence $\lim_{M\to \infty} \limsup_{n\to \infty} \mathbb{P}(A_n B_n(M) C_n >\eta) =0$ for all $\eta > 0$. Again, with \eqref{incr3}, we derive for $M\geq 1$
\begin{align*}
&\Big| \frac{\sqrt{n}}{n^l}\sum_{k=1}^{\left\lfloor nt\right\rfloor}\sum_{\textbf{i}\in \mathcal{B}_t^n (l-1)}\sum_{\textbf{j}\in \mathcal{B}_t^n (d-l)}\mathbbm{1}_{\left\{\left\|\sqrt{n}\Delta_{\mathbf{i}}^n X^c\right\| > M \right\}}\mathbbm{1}_{(B_k^n)^C}\Big( H(\sqrt{n}\Delta_k^n X^c,\sqrt{n}\Delta_{\mathbf{i}}^n X^c,\Delta_{\mathbf{j}}^n X) -H(\alpha_k^n,\sqrt{n}\Delta_{\mathbf{i}}^n X^c,\Delta_{\mathbf{j}}^n X) \Big) \Big| \\
&=  \frac{\sqrt{n}}{n^l}\sum_{k=1}^{\left\lfloor nt\right\rfloor}\sum_{\textbf{i}\in \mathcal{B}_t^n (l-1)}\sum_{\textbf{j}\in \mathcal{B}_t^n (d-l)}\mathbbm{1}_{\left\{\left\|\sqrt{n}\Delta_{\mathbf{i}}^n X^c\right\| > M \right\}}\mathbbm{1}_{(B_k^n)^C} |\Delta_{\mathbf{j}}^n X|^{\mathbf{q}}|\sqrt{n}\Delta_{\mathbf{i}}^n X^c|^{\tilde{\mathbf{p}}} \\
& \quad \quad \quad \quad \quad \quad \quad \quad \quad \quad \quad \quad \quad \quad \quad \quad \quad \quad \times \Big| g(\alpha_k^n +\theta_k^n,\sqrt{n}\Delta_{\mathbf{i}}^n X^c,\Delta_{\mathbf{j}}^n X) -g(\alpha_k^n,\sqrt{n}\Delta_{\mathbf{i}}^n X^c,\Delta_{\mathbf{j}}^n X)  \Big|\\
& \leq K\Big(\sum_{\textbf{j}\in \mathcal{B}_t^n (d-l)} |\Delta_{\mathbf{j}}^n X|^{\mathbf{q}} \Big) \Big(\frac{1}{n^{l-1}} \sum_{\textbf{i}\in \mathcal{B}_t^n (l-1)}\mathbbm{1}_{\left\{\left\|\sqrt{n}\Delta_{\mathbf{i}}^n X^c\right\| > M \right\}} |\sqrt{n}\Delta_{\mathbf{i}}^n X^c|^{\tilde{\mathbf{p}}}\left\|\sqrt{n}\Delta_{\mathbf{i}}^n X^c\right\|^{h''}\Big)\\
& \quad \quad \quad \quad \quad \quad \quad \quad \quad \quad \quad \quad \quad \quad \quad \quad \quad \quad \times \Big(\frac{1}{\sqrt{n}}\sum_{k=1}^{\left\lfloor nt\right\rfloor} \mathbbm{1}_{(B_k^n)^C}(1+|\alpha_k^n|^{h''})|\theta_k^n|^{p_1}\Big)\\
&\leq K \Big(\sum_{\textbf{j}\in \mathcal{B}_t^n (d-l)} |\Delta_{\mathbf{j}}^n X|^{\mathbf{q}} \Big) \Big(\frac{1}{n^{l-1}} \sum_{\textbf{i}\in \mathcal{B}_t^n (l-1)}\mathbbm{1}_{\left\{\left\|\sqrt{n}\Delta_{\mathbf{i}}^n X^c\right\| > M \right\}} |\sqrt{n}\Delta_{\mathbf{i}}^n X^c|^{\tilde{\mathbf{p}}}\left\|\sqrt{n}\Delta_{\mathbf{i}}^n X^c\right\|^{h''}\Big)\\
& \quad \quad \quad \quad \quad \quad \quad \quad \quad \quad \quad \quad \quad \quad \quad \quad \quad \quad \times \Big(\frac{1}{\sqrt{n}}\sum_{k=1}^{\left\lfloor nt\right\rfloor} (1+|\alpha_k^n|^{h''})|\alpha_k^n|^{p_1-1}|\theta_k^n|\Big).
\end{align*}
For the last step, recall that $|\theta_k^n|^{1-p_1} \leq K |\alpha_k^n|^{1-p_1}$ on the set $(B_k^n)^C$. Once again, the final random variable converges to $0$ if we first let $n$ and then $M$ to infinity. \qed

\noindent \textit{Proof of Proposition \ref{appsp}:}
We will give a proof only in the case $d=2$ and $l=1$. We use the decomposition
\begin{align*}
&\mathbbm{1}_{\Omega_n(m)}\theta_n(H) \\
=&  \frac{\mathbbm{1}_{\Omega_n(m)}}{\sqrt{n}}\Big(\sum_{i,j=1}^{\left\lfloor nt\right\rfloor} H(\alpha_i^n, \Delta_j^n X(m)) -\sum_{i=1}^{\left\lfloor nt\right\rfloor} \sum_{s \leq \frac{\left\lfloor nt\right\rfloor}{n}} H(\alpha_i^n, \Delta X(m)_s) \Big)- \frac{\mathbbm{1}_{\Omega_n(m)}}{\sqrt{n}}\sum_{i=1}^{\left\lfloor nt\right\rfloor} \sum_{\frac{\left\lfloor nt\right\rfloor}{n}<s \leq t} H(\alpha_i^n, \Delta X_s) \\
+ & \frac{\mathbbm{1}_{\Omega_n(m)}}{\sqrt{n}} \sum_{i=1}^{\left\lfloor nt \right\rfloor} \sum_{p \in \mathcal{P}_t^n(m)} \left\{ H\big(\alpha_i^n, \Delta X_{S_p}+n^{-1/2}R(n,p)\big)-H\big(\alpha_i^n, n^{-1/2}R(n,p)\big)\right\} - \frac{\mathbbm{1}_{\Omega_n(m)}}{\sqrt{n}}\sum_{i=1}^{\left\lfloor nt \right\rfloor} \sum_{p \in \mathcal{P}_t^n(m)} H(\alpha_i^n, \Delta X_{S_p}) \\
=:& \theta_n^{(1)}(H)-\theta_n^{(2)}(H)+\theta_n^{(3)}(H)-\theta_n^{(4)}(H).  
\end{align*}
In the general case we would have to use the decomposition given in \eqref{dcom} for the last $d-l$ arguments.
We first show that we have 
\begin{align}\label{appju}
\lim_{m \to \infty} \limsup_{n \to \infty} \mathbb{P}(|\theta_n^{(1)}(H)| >\eta)=0 \quad \text{for all} \quad \eta >0. 
\end{align}
We do this in two steps. 

\noindent a) Let $\phi_k$ be a function in $\mathcal{C}^{\infty}(\mathbb{R}^2)$ with $0\leq \phi_k \leq 1$, $\phi_k \equiv 1$ on $[-k,k]^2$, and $\phi_k \equiv 0$ outside of $[-2k,2k]^2$. Also, let $\tilde{g}:\mathbb{R}^2 \to \mathbb{R}$ be defined by $\tilde{g}(x,y)=|y|^{q_1}L(x,y)$ and set $H_k=\phi_k H$ and $\tilde{g_k}=\phi_k \tilde{g}$. Then we have
\begin{align*}
|\theta_n^{(1)}(H_k)|&= \Big|\frac{\mathbbm{1}_{\Omega_n(m)}}{\sqrt{n}} \sum_{i=1}^{\left\lfloor nt\right\rfloor} |\alpha_i^n|^{p_1}\Big(\sum_{j=1}^{\left\lfloor nt\right\rfloor} \tilde{g_k}(\alpha_i^n , \Delta_j^n X(m)) - \sum_{s \leq \frac{\left\lfloor nt\right\rfloor}{n}} \tilde{g_k}(\alpha_i^n , \Delta X(m)_s )\Big)    \Big|\\
&\leq \Big|\frac{\mathbbm{1}_{\Omega_n(m)}}{\sqrt{n}} \sum_{i=1}^{\left\lfloor nt\right\rfloor} |\alpha_i^n|^{p_1}\Big(\sum_{j=1}^{\left\lfloor nt\right\rfloor} \tilde{g_k}(0 , \Delta_j^n X(m)) - \sum_{s \leq \frac{\left\lfloor nt\right\rfloor}{n}} \tilde{g_k}(0 , \Delta X(m)_s )\Big)    \Big| \\
& \quad \quad +\Big|\frac{\mathbbm{1}_{\Omega_n(m)}}{\sqrt{n}} \sum_{i=1}^{\left\lfloor nt\right\rfloor} |\alpha_i^n|^{p_1}\Big(\sum_{j=1}^{\left\lfloor nt\right\rfloor} \int_0^{\alpha_i^n}\partial_1\tilde{g_k}(u , \Delta_j^n X(m)) du- \sum_{s \leq \frac{\left\lfloor nt\right\rfloor}{n}} \int_0^{\alpha_i^n} \partial_1\tilde{g_k}(u , \Delta X(m)_s ) du\Big)    \Big| \\
& \leq \Big(\frac{1}{n} \sum_{i=1}^{\left\lfloor nt\right\rfloor} |\alpha_i^n|^{p_1} \Big) \Big( \sqrt{n}\mathbbm{1}_{\Omega_n(m)}\Big|\sum_{j=1}^{\left\lfloor nt\right\rfloor} \tilde{g}_k(0 , \Delta_j^n X(m)) - \sum_{s \leq \frac{\left\lfloor nt\right\rfloor}{n}} \tilde{g}_k(0 , \Delta X(m)_s )\Big| \Big)  \\
& \quad \quad +\Big(\frac{1}{n} \sum_{i=1}^{\left\lfloor nt\right\rfloor} |\alpha_i^n|^{p_1}\Big) \Big(\mathbbm{1}_{\Omega_n(m)} \int_{-k}^{k} \sqrt{n}\Big|\sum_{j=1}^{\left\lfloor nt\right\rfloor}\partial_1\tilde{g}_k(u , \Delta_j^n X(m)) - \sum_{s \leq \frac{\left\lfloor nt\right\rfloor}{n}}  \partial_1\tilde{g}_k(u , \Delta X(m)_s )\Big| du\Big) ,   
\end{align*}
which converges to zero in probability by Lemma \ref{lem2}, if we first let $n \to \infty$ and then $m\to \infty$, since 
\[
\frac{1}{n} \sum_{i=1}^{\left\lfloor nt\right\rfloor} |\alpha_i^n|^{p_1}
\]
is bounded in probability by Burkholder inequality.

\noindent b) In this part we show
\[
\lim_{k \to \infty} \lim_{m\to \infty} \limsup_{n \to \infty} \mathbb{P}(|\theta_n^{(1)}(H)-\theta_n^{(1)}(H_k)| >\eta )=0 \quad \text{for all} \quad \eta >0.
\]
Observe that we automatically have $|\Delta_i^n X(m)| \leq k$ for some $k$ large enough. Therefore,
\begin{align*}
|\theta_n^{(1)}(H)-\theta_n^{(1)}(H_k)|&= \Big| \frac{\mathbbm{1}_{\Omega_n(m)}}{\sqrt{n}} \sum_{i,j=1}^{\left\lfloor nt\right\rfloor}\Big( H(\alpha_i^n,\Delta_j^n X(m))- H_k (\alpha_i^n,\Delta_j^n X(m))\Big) \Big| \\
&\leq \frac{\mathbbm{1}_{\Omega_n(m)}}{\sqrt{n}}  \sum_{i,j=1}^{\left\lfloor nt\right\rfloor}\mathbbm{1}_{\left\{ |\alpha_i^n|>k\right\}}\Big| H(\alpha_i^n,\Delta_j^n X(m))- H_k (\alpha_i^n,\Delta_j^n X(m))\Big|\\
& \leq \frac{\mathbbm{1}_{\Omega_n(m)}}{\sqrt{n}}  \sum_{i,j=1}^{\left\lfloor nt\right\rfloor}\mathbbm{1}_{\left\{ |\alpha_i^n|>k\right\}}\Big| H(\alpha_i^n,\Delta_j^n X(m))\Big| \\
& \leq \frac{K\mathbbm{1}_{\Omega_n(m)}}{\sqrt{n}}  \sum_{i,j=1}^{\left\lfloor nt\right\rfloor}\mathbbm{1}_{\left\{ |\alpha_i^n|>k\right\}} \big|(1+|\alpha_i^n|^{p_1})( \Delta_j^n X(m) )^{q_1}\big| \\
& \leq \frac{K}{\sqrt{n}}\Big(\sum_{i=1}^{\left\lfloor nt\right\rfloor}\mathbbm{1}_{\left\{ |\alpha_i^n|>k\right\}} (1+|\alpha_i^n|^{p_1})\Big) \Big(\mathbbm{1}_{\Omega_n(m)}\sum_{j=1}^{\left\lfloor nt\right\rfloor} \big| \Delta_j^n X(m) \big|^{q_1} \Big)\\
& \leq K\Big(\sum_{i=1}^{\left\lfloor nt\right\rfloor}\mathbbm{1}_{\left\{ |\alpha_i^n|>k\right\}}\Big)^{\frac{1}{2}} \Big(\frac{1}{n}\sum_{i=1}^{\left\lfloor nt\right\rfloor}(1+|\alpha_i^n|^{p_1})^2\Big)^{\frac{1}{2}} \Big(\mathbbm{1}_{\Omega_n(m)}\sum_{j=1}^{\left\lfloor nt\right\rfloor} \big| \Delta_j^n X(m) \big|^{q_1} \Big)
\end{align*}
Now observe that we have
\[
\Big(\mathbbm{1}_{\Omega_n(m)}\sum_{j=1}^{\left\lfloor nt\right\rfloor} \big| \Delta_j^n X(m) \big|^{q_1} \Big) \toop \sum_{s \leq t} |\Delta X_s|^{q_1},
\] 
if we first let $n \to \infty$ and then $m \to \infty$. Further we have  
\[
\mathbb{E} \big[ \frac{1}{n}\sum_{i=1}^{\left\lfloor nt\right\rfloor}(1+|\alpha_i^n|^{p_1})^2 \big]\leq K
\]
by Burkholder inequality and finally
\begin{align*}
\mathbb{P}\Big( \Big| \sum_{i=1}^{\left\lfloor nt\right\rfloor}\mathbbm{1}_{\left\{ |\alpha_i^n|>k\right\}}  \Big| > \eta \Big) &\leq \frac{1}{\eta}\mathbb{E}\Big( \sum_{i=1}^{\left\lfloor nt\right\rfloor}\mathbbm{1}_{\left\{ |\alpha_i^n|>k\right\}} \Big) \leq \sum_{i=1}^{\left\lfloor nt\right\rfloor} \frac{\mathbb{E}[|\alpha_i^n|^2]}{\eta k^2} \leq \frac{K}{\eta k^2} \to 0,
\end{align*}
as $k \to \infty$. For $\theta_n^{(2)}(H)$ we have
\begin{align*}
|\theta_n^{(2)}(H)| \leq \frac{1}{\sqrt{n}}\sum_{i=1}^{\left\lfloor nt\right\rfloor} \sum_{\frac{\left\lfloor nt\right\rfloor}{n}<s \leq t} (1+|\alpha_i^n|^{p_1})|\Delta X_s |^{q_1} u(\Delta X_s) \leq \Big( \frac{1}{n}\sum_{i=1}^{\left\lfloor nt\right\rfloor}  (1+|\alpha_i^n|^{p_1})\Big)\Big(\sqrt{n}\sum_{\frac{\left\lfloor nt\right\rfloor}{n}<s \leq t}|\Delta X_s |^{q_1} \Big)\toop 0, 
\end{align*}
since the first factor is bounded in expectation and the second one converges in probability to $0$ (see \eqref{rest}). For the second summand of $\theta_n^{(3)}(H)$ we get
\begin{align*}
\Big|\frac{\mathbbm{1}_{\Omega_n(m)}}{\sqrt{n}} \sum_{i=1}^{\left\lfloor nt \right\rfloor} \sum_{p \in \mathcal{P}_t^n(m)} H\big(\alpha_i^n, n^{-1/2}R(n,p)\big)\Big| \leq \Big( \frac{1}{n}\sum_{i=1}^{\left\lfloor nt\right\rfloor}  (1+|\alpha_i^n|^{p_1})\Big)\Big(\mathbbm{1}_{\Omega_n(m)}\sum_{p \in \mathcal{P}_t^n(m)}\Big|\frac{R(n,p)^{q_1}}{n^{\frac{1}{2}(q_1-1)}}\Big|   \Big) \toop 0
\end{align*}
 as $n \to \infty$ because the first factor is again bounded in expectation and since $(R(n,p))_{n \in \mathbb{N}}$ is bounded in probability and $\mathcal{P}_t^n(m)$ finite almost surely. The remaining terms are $\theta_n^{(4)}(H)$ and the first summand of $\theta_n^{(3)}(H)$, for which we find by the mean value theorem
\begin{align*}
 \frac{\mathbbm{1}_{\Omega_n(m)}}{\sqrt{n}} \sum_{i=1}^{\left\lfloor nt \right\rfloor} \sum_{p \in \mathcal{P}_t^n(m)} \left\{ H\big(\alpha_i^n, \Delta X_{S_p}+n^{-1/2}R(n,p)\big)-  H(\alpha_i^n, \Delta X_{S_p})\right\} =\frac{\mathbbm{1}_{\Omega_n(m)}}{n} \sum_{i=1}^{\left\lfloor nt \right\rfloor} \sum_{p \in \mathcal{P}_t^n(m)} \partial_2 H\big(\alpha_i^n, \Delta X_{S_p}\big)R(n,p) \\
  + \Big( \frac{\mathbbm{1}_{\Omega_n(m)}}{n} \sum_{i=1}^{\left\lfloor nt \right\rfloor} \sum_{p \in \mathcal{P}_t^n(m)} \big(\partial_2 H\big(\alpha_i^n, \Delta X_{S_p}+\xi_i^n (p)\big) - \partial_2 H\big(\alpha_i^n, \Delta X_{S_p}\big) \big)R(n,p)\Big)
\end{align*}
for some $\xi_i^n(p)$ between $0$ and $R(n,p)/\sqrt{n}$. The latter term converges to $0$ in probability since we have $|\partial_{22} H(x,y)| \leq (1+|x|^{q}) (|y|^{q_1} +|y|^{q_1 -1} + |y|^{q_1-2}) u(y)$ for some $q\geq 0$ by the growth assumptions on $L$. Therefore, 
\begin{align*}
&\Big| \frac{\mathbbm{1}_{\Omega_n(m)}}{n} \sum_{i=1}^{\left\lfloor nt \right\rfloor} \sum_{p \in \mathcal{P}_t^n(m)} \big(\partial_2 H\big(\alpha_i^n, \Delta X_{S_p}+\xi_i^n (p)\big) - \partial_2 H\big(\alpha_i^n, \Delta X_{S_p}\big) \big)R(n,p)\Big| \\
=& \Big| \frac{\mathbbm{1}_{\Omega_n(m)}}{n} \sum_{i=1}^{\left\lfloor nt \right\rfloor} \sum_{p \in \mathcal{P}_t^n(m)} \partial_{22} H\big(\alpha_i^n, \Delta X_{S_p}+\tilde{\xi}_i^n (p)\big)  \xi_i^n (p) R(n,p)\Big| \\
\leq & \Big(\frac{1}{n}\sum_{i=1}^{\left\lfloor nt\right\rfloor}  (1+|\alpha_i^n|^{p_1}) \Big) \sum_{p \in \mathcal{P}_t^n(m)} K\frac{|R(n,p)|^2}{\sqrt{n}} \toop 0,
\end{align*}
where $\tilde{\xi}_i^n (p)$ is between $0$ and $R(n,p)/\sqrt{n}$. The last inequality holds since the jumps of $X$ are bounded and $|\tilde{\xi}_i^n (p)|\leq |R(n,p)|/\sqrt{n} \leq 2A$. The convergence holds because $R(n,p)$ is bounded in probability and $ \mathcal{P}_t^n(m)$ is finite almost surely. \qed
\end{subsection}
\end{section}

\end{document}